\documentclass{article}
\usepackage[utf8]{inputenc}

\providecommand{\main}{.}

\usepackage{a4wide}
\usepackage[ruled,vlined]{algorithm2e}
\usepackage{lmodern}
\usepackage[utf8]{inputenc}
\usepackage[T1]{fontenc}
\usepackage{latexsym,exscale,amssymb,amsmath}
\usepackage{graphicx}
\usepackage[nointegrals]{wasysym}
\usepackage{eurosym}
\usepackage{amsthm}
\usepackage{hyperref}
\usepackage{caption}
\usepackage{svg}
\usepackage{blindtext}
\usepackage{bbm}
\usepackage{subcaption}
\usepackage{tikz}
\usepackage[nomain]{glossaries}
\usepackage{authblk}

\usetikzlibrary{shapes.geometric, arrows, arrows.meta}
\tikzstyle{process} = [rectangle, 
minimum width=2.5cm, 
minimum height=1cm, 
text centered, 
text width=3cm, 
rounded corners=0.5cm,
draw=black, 
fill=lightgray!50]

\tikzstyle{arrow} = [-{Implies},double, line width=.4mm]
\usepackage{subfiles} 

\graphicspath{{\main/images/}{images/}}

\DeclareMathOperator*{\argmin}{argmin}
\DeclareMathOperator*{\cov}{Cov}
\DeclareMathOperator*{\supp}{supp}
\DeclareMathOperator*{\diam}{diam}
\DeclareMathOperator{\dir}{dir}

\DeclareMathOperator*{\TV}{TV}

\DeclareMathOperator{\arcosh}{arcosh}
\DeclareMathOperator*{\cl}{cl}

\newcommand{\cO}{\mathcal{O}}
\newcommand{\cP}{\mathcal{P}}
\newcommand{\cQ}{\mathcal{Q}}
\newcommand{\cT}{\mathcal{T}}

\newcommand{\fb}{\mathfrak{b}}

\newcommand{\prb}{P}
\newcommand{\prbalt}{Q}
\newcommand{\NN}{{\mathbb N}}
\newcommand{\RR}{{\mathbb R}}
\newcommand{\EE}{{\mathbb E}}
\newcommand{\mean}{\mu}
\newcommand{\meanalt}{\nu}
\newcommand{\smean}{\hat{\mean}_n}
\newcommand{\diff}{\mathrm{d}}
\newcommand{\wst}[2][1]{\mathcal{P}^{#1}\left(#2\right)}
\newcommand{\Qkap}[1][1]{\mathcal{Q}^{#1}_{x, \kappa}}
\newcommand{\Qkapn}[1][1]{\mathcal{Q}^{#1}_{x_0, \kappa}}
\newcommand{\Bkapf}{B_{R_\kappa}(x)}

\newcommand{\Bkap}{B_{R_\kappa/2}(x)}
\newcommand{\Bkapn}{B_{R_\kappa/2}(x_0)}
\newcommand{\iid}{{\stackrel{i.i.d.}{\sim}}}
\newcommand{\cat}{CAT$(\kappa)$}
\newcommand{\catzero}{CAT$(0)$}

\newtheorem{theorem}{Theorem}[section]
\newtheorem{proposition}[theorem]{Proposition}
\newtheorem{lemma}[theorem]{Lemma}
\newtheorem{corollary}[theorem]{Corollary}
\newtheorem{setting}[]{Scenario}
\newtheorem{conjecture}[theorem]{Conjecture}

\theoremstyle{definition}
\newtheorem{definition}[theorem]{Definition}

\theoremstyle{remark}
\newtheorem{remark}[theorem]{Remark}
\newtheorem{example}[theorem]{Example}

\usepackage[backend=biber, style = alphabetic]{biblatex}
\begin{document}
\title{
	Sticky Flavors
	\thanks{Supported by DFG GK 2088 and DFG HU 1575/7.}}
\author{Lars Lammers, Do Tran Van and Stephan F. Huckemann }

\affil{Felix-Bernstein-Institute for Mathematical Statistics
	in the Biosciences, University of G\"ottingen,
	Goldschmidtstrasse 7, 37077 G\"ottingen, Germany}

\maketitle              

\begin{abstract}
	The Fr\'echet mean, a generalization to a metric space of the expectation of a random variable in a vector space, can exhibit unexpected behavior for a wide class of random variables. For instance, it can stick to a point (more generally to a closed set) under resampling: sample stickiness. It can stick to a point for topologically nearby distributions: topological stickiness, such as total variation or Wasserstein stickiness. It can stick to a point for slight but arbitrary perturbations: perturbation stickiness. Here, we explore these and various other flavors of stickiness and their relationship in varying scenarios, for instance on \cat~spaces, $\kappa\in \RR$. Interestingly, modulation stickiness
	(faster asymptotic rate than $\sqrt{n}$)
	and directional stickiness (a generalization of moment stickiness from the literature) allow for the development of new statistical methods building on an asymptotic fluctuation, where, due to stickiness, the mean itself features no asymptotic fluctuation. Also, we rule out sticky flavors on manifolds in scenarios with curvature bounds.

\end{abstract}
\hspace{10pt}

\tableofcontents

\section{Introduction}

For this entire paper, let $(M,d)$ be a complete and separable metric space, $\cP(M)$ denotes the space of Borel probability distributions on $M$ and for $q\geq 1$ consider
\begin{align*}
	\wst[q]{M} :=  \Big \{ \prb \in \mathcal{P}(M) \ \Big \vert \ \exists  x \in M: \  \int_M d^q(x, z) \ \diff \prb(z) < \infty \Big \}\,.
\end{align*}
This gives rise to the \emph{Fr\'echet function}
\begin{align*}
	F_\prb(x) := \frac{1}{2} \int_M d^2(x, y) \  \diff \prb(y), \quad x \in M
\end{align*}
for $\prb\in \wst[2]{M}$ and  its \emph{Fr\'echet mean} set
\begin{align*}
	\mathfrak{b}(\prb) := \argmin_{x \in M} F_\prb(x)\,.
\end{align*}
Due to an observation noted in \cite[p. 33]{Sturm}, detailed in Remark \ref{rmk:Frechet-diff}, since for $\prb\in \wst[2]{P}$, every minimizer of $x\mapsto F_\prb(x)$ is also a minimizer of $x \mapsto F_\prb(x)-F_\prb(y)$ for arbitrary fixed $y\in M$, this difference can be given a meaning, even if the two terms diverge, due to the triangle inequality. Hence, we consider Fr\'echet means for all $\prb \in \wst[1]{M}$.

Another observation is that the Fr\'echet mean set can \emph{stick} to a closed set $S\subset M$, of lower dimension, on spaces with unbounded negative Alexandrov curvature, like \emph{open books} \cite{HHL+13}, BHV spaces \cite{BHV02}, \cite{bhvclt1,bhvclt2} 
and the \emph{kale} \cite{HMMN15}. A simple illustration of this phenomenon gives Example \ref{ex:kspider}. On the kale stickiness for $S$ being a single point, was phrased in terms of \emph{moment}, \emph{Wasserstein}, \emph{perturbation} and \emph{sample stickiness}, see Example \ref{ex:kale},  and there, \cite{HMMN15} showed that these four \emph{flavors} of stickiness are equivalent. Recently, equivalence in a slightly more general scenario was also shown by \cite{bhvsticky} for BHV spaces.

In this paper, we carefully introduce the above-mentioned \emph{flavors of stickiness} in Section \ref{scn:3-flavors}, and more flavors in Sections \ref{scn:dir-sticky} and \ref{scn:mod-sticky-asymp}, and relate them to one another in specific \emph{scenarios}, involving, among others stickiness with respect to certain families $\cQ \subseteq \wst[q]M$.

Notably, within our scenarios, in Section \ref{scn:rule-out} we can rule out stickiness on manifolds.

Also, we derive in Section \ref{scn:CLT-dir} a central limit theorem for the \emph{directions of stickiness}, which arise from \emph{directional stickiness}, which is a generalization of \emph{moment stickiness} from \cite{HHL+13,HMMN15}, cf. Remark \ref{remark:folded-moments-directional-derivatives}, and we provide asymptotic bounds for \emph{moment modulations} in case of \emph{modulation stickiness} in Section \ref{scn:mod-sticky-asymp}. 

This contribution contains several examples and counterexamples for easy tasting our sticky flavors.

We conclude this introduction by giving first, an overview of \emph{sticky flavors} and thirdly, their rather involved relationship in, secondly, varying \emph{scenarios}.

\subsection{Flavors of Stickiness}
For a measurable set $A\subseteq M$ and $q\geq 1$ let
$$ \cQ^q_A := \{\prb \in \wst[q]{M}: \supp(\prb) \subseteq A\}\,.$$
Moreover, for a random variable $X$ with probability distribution $\prb \in \cQ^q_A$,  $\prb_n \in \cQ^q_A$ denotes the probability distribution of a sample $X_1,\ldots,X_n \iid X$, $n\in \NN$. Recalling that $S \subset M$ is closed, we consider the following sticky flavors.
\begin{description}
	\item[perturbation sticky]  (P): for every $y\in A$ there is $t_y >0$ such that   $\mathfrak{b}((1-t)P+t\delta_y)\subseteq S$ for all $0\leq t \leq t_y$;
	\item[sample sticky] (S): there is a random integer $N$ such that $\mathfrak{b}(P_n) \subseteq S$ a.s. for all $n\geq N$;
	\item[topologically sticky] (T): there is a neighborhood $U$ of $\prb$ in $\cQ^q_A$  such that $\mathfrak{b}(\prbalt) \subseteq S$ for all $\prbalt\in U$; typical topologies for $\cQ^q_A$ are Wasserstein topologies leading to \emph{Wasserstein sticky} ($W_q$), or topologies induced by $f$-divergences leading to \emph{$f$-divergens sticky} ($D_f$) such as \emph{total variation sticky} (TV) or \emph{Kullback-Leibler sticky} (KL), say, detailed in Section \ref{scn:3-flavors};
	\item[directional sticky] ($\nabla$): if $S=\{\mu\}$ for some $\mu \in M$ and $\nabla_\sigma F_P(\mu)>0$ for all $\sigma\in \Sigma_\mu$ with  the \emph{space of directions} $\Sigma_\mu$ at $\mu$, detailed in Section \ref{scn:dir-sticky};
	\item[modulation sticky]  ($\mathfrak{m}_q$): if  $S=\{\mu\}$ for some $\mu \in M$ and ${\sqrt{n}}\,^q\EE[d(\mu,\hat \mu_n)^2] \to 0$ as $n\to \infty$ where  $\hat\mu_n$ is a measurable selection of $\mathfrak{b}(P_n)$, Section \ref{scn:mod-sticky} details that we actually consider this on the \emph{tangent cone} $T_\mu$.
\end{description}

\subsection{Four Scenarios}

For $\prb \in \cQ^q_A$ \emph{sticking} within $\cQ^q_A$ to a closed set $S\subset M$, consider as follows.
\begin{description}
	\item[Scenario I:]  $(M,d)$ is a complete metric space;
	\item[Scenario II:] in addition to Scenario I, $S=\{\mu\}$ for some $\mu \in M$, $(M,d)$ is CAT($\kappa$) for some  $\kappa \in \RR$ (i.e. it features \emph{Alexandrov curvatures} bounded from above by $\kappa $) and $A$ is a \emph{geodesic half ball} $A=B_{R_\kappa/2}^{x_0}$ containing $\mu$ and centered at some $x_0\in M$ (since $R_\kappa = \infty$ for $\kappa \leq 0$, see Equation (\ref{eq:R_kappa}), in that case $A=M$ and $\cQ^q_A = \wst[q]{M}$), all of this is detailed in Sections \ref{scn:CAT} and \ref{scn:dir-sticky};
	\item[Scenario III:] in addition to Scenario II $(M,d)$ is geodesically complete and locally compact, see Section \ref{scn:mod-sticky};
	\item[Scenario IV:]  in addition to Scenario III at least second moments exist and $\Sigma_\mu$ has a polynomial rate \emph{covering number} (which is the case, for instance, if it is finite dimensional), all of this is detailed in Section \ref{scn:mod-sticky-asymp}.
\end{description}
We also require a
\begin{description}
	\item[pull condition] (pc): the \emph{pull} $M\to \RR, z\mapsto \phi_{\sigma,\mu}(z)$ does not vanish a.s. for any direction $\sigma \in \Sigma_\mu$, in the context of \emph{directional stickiness},  detailed in Section \ref{scn:pull}, and a
	\item[$f$-condition] (fc): $f$ is twice differentiable and $f''(x) > 0$ for all $x\neq 1$ in the context of \emph{$f$-divergence stickiness}, detailed in Section \ref{scn:3-flavors}.
\end{description}

\subsection{Relating the Flavors to One Another}

The rather involved relationship between sticky flavors for varying scenarios is illustrated in the following Figure \ref{fig:stickychart}.

\begin{figure}[h!]
	\centering
	\begin{tikzpicture}[node distance=2cm]
		\node (w1) [process, yshift = 1.5cm] {$W_1$-sticky};
		\node (df) [process, left of=w1, xshift=-3.7cm, yshift= -1.75cm] {$D_f$-sticky};
		\node (tv) [process, left of=w1, xshift=-3.7cm] {$\TV$-sticky};
		\node (wq) [process, above of=w1,yshift = 1.5cm] {$W_q$-sticky};
		\node (sample) [process, right of=w1, xshift=3.7cm] {sample sticky};
		\node (dir) [process, below of=w1, yshift=-1.5cm] {directionally sticky};
		\node (pert) [process, left of=dir, xshift=-3.7cm] {perturbation sticky};
		\node (mod) [process, right of=dir, xshift=3.7cm] {$q$-modulation stickiness};

		\draw [arrow, black] (w1) to node[anchor= north, xshift=.0cm,align=center] {*}(pert);
		\draw [arrow, black] (w1) to (wq);
		\draw [arrow, black] (wq) to (sample);
		\draw [arrow, black] (wq.west) |-([shift={(-65mm,0)}]wq.west)-- ([shift={(-8mm,0mm)}]pert.west)--(pert.west);
		\draw [arrow, black] (df) to node[anchor= west, xshift=.0cm,align=center] {*}(pert);
		\draw [arrow, black] (w1) to node[anchor= south, xshift=.0cm,align=center] {*}(sample);
		\draw [arrow, black] (w1) to node[anchor=south,xshift=.0cm,align=center] {$A$ bounded}(tv);
		\draw [arrow, black] (tv) to node[anchor= east, xshift=.0cm,align=center] {(fc)} (df);
		\draw [arrow, teal] (sample) to node[anchor=south east,yshift=-0.1cm,align=center] {(pc)}node[anchor= north west, xshift=.0cm,align=center] {\textcolor{black}{*}}(dir);
		\draw [arrow, teal] (pert) to (dir);
		\draw [arrow, blue!70] (dir) to (w1);
		\draw [arrow, bend angle = 8, bend right, magenta] (dir) to node[anchor=north,yshift=0cm,align=center] {$\prb \in \wst[2 \lor (q + \epsilon)]{M}$} (mod);
		\draw [arrow, bend angle = 8, bend right, magenta] (mod) to node[anchor=south,yshift=0cm,align=center] {(PC),$\prb \in \wst[2]{M}$}(dir);

		\path ([xshift=-3cm,yshift=0cm]current bounding box.north east)
		node[matrix,anchor=north west,cells={nodes={font=\sffamily,anchor=west}},
			draw,thick,inner sep=1ex]{
			\draw [arrow, black](0,0) -- (0.6,0); & \node{Scenario \ref{set:1}};\\
			\draw [arrow, teal](0,0) -- ++ (0.6,0); & \node{Scenario \ref{set:2}};\\
			\draw [arrow, blue!70](0,0) -- ++ (0.6,0); & \node{Scenario \ref{set:3}};\\
			\draw [arrow, magenta](0,0) -- ++ (0.6,0); & \node{Scenario \ref{set:4}};\\
		};
	\end{tikzpicture}
	\caption{\it Relationships between the various flavors of stickiness. For every arrow with an asterisk (*) we give a counterexample for either the converse direction or
		in a different scenario.}
	\label{fig:stickychart}
\end{figure}
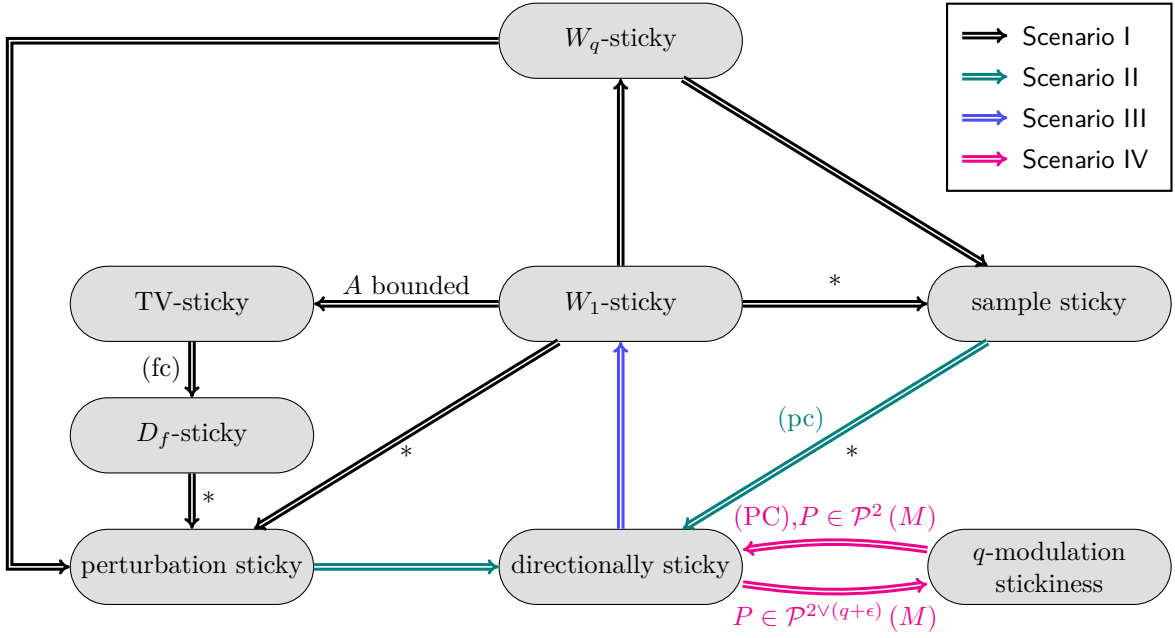

\subsection{Notation and a Remark}

For $x \in M$ and $r >0$,
$$B_r(x) := \{y\in M: d(x,y) < r\}\,,$$
is the open ball around $x$ of radius $r$, where for $r=\infty$ set $B_r(x)=M$.

The support of a probability distribution $\prb \in \cP (M)$ is
given by
$$ \supp(\prb) = \cap_{A\subseteq M~ \mbox{closed}, \prb(A) = 1} A$$

For $x\in M$, the point measure at $x$ is $\delta_x \in \cP(M)$ and $$\prb_n = \frac{1}{n}
	\sum_{j=1}^n \delta_{X_j}$$
is the \emph{empirical measure} of a sample $X_1,\ldots X_n  \iid \prb \in  \cP(M)$, $n\in  \NN$.

Beginning with Section \ref{scn:dir-sticky}, $(M,d)$ is a \cat~space with $\kappa \in \RR$ and for $q \geq 1$, $x \in M$,
\begin{align*}
	\Qkap[q] = \Big \{ \prb \in \wst[q]{M} \ : \ \supp(\prb) \subset \Bkap \Big \},
\end{align*}
is the space of probability measures with finite $q$-th moments, supported by an open geodesic half ball, which is all of $M$ if $\kappa \leq 0$.

The metric completions of the space of directions $\Sigma_x$ and of the tangent cone $T_x$ at $x\in M$ are $\cl(\Sigma_x)$ and $\cl(T_x)$, respectively.

\begin{remark}
	\label{rmk:Frechet-diff}

	Indeed, as \cite[p. 33]{Sturm} noted, it suffices to consider $\prb\in \wst{M}$, since for $\prb\in \wst[2]{M}$, minimizers of $F_\prb$ agree with minimizers of the
	\emph{Fr\'echet difference} given, with arbitrary $y\in M$,  by
	\begin{align*}
		F_\prb^y(x) = \frac{1}{2} \int_M \left(d^2(x,z) - d^2(y,z)\right) \
		\diff \prb(z), \quad x \in M\,,
	\end{align*}
	which is, due to the triangle inequality
	\begin{align*}
		\lvert d^2(x,z) - d^2(y,z) \rvert & = \lvert d(x,z) - d(y,z)\rvert \cdot \big(d(x,z) + d(y,z)\big) \\
		                                  & \leq  d(x,y) \cdot \big(d(x,z) + d(y,z)\big),
	\end{align*}
	already finite for $\prb\in \wst{M}$.
\end{remark}

\section{Topological, Perturbation and Sample Stickiness}  
\label{scn:3-flavors}

Here we introduce the three different basic flavors of stickiness, where the first one forks into several subflavors, and we study some immediate relationships between them in the very general \emph{Scenario I}. In Sections \ref{scn:dir-sticky} and \ref{scn:mod-sticky-asymp}, we will add a fourth (directional) and a fifth (modulation) flavor, respectively, that require more structure of the metric space, namely a log map to its tangent cones.

In particular, the tangent cone is a \emph{Euclidean cone} as introduced below and such cones play not only a central role in mathematical reasoning throughout this paper, they were also the first spaces on which stickiness had been observed. We start with this.

%

\subsection{Euclidean Cones and Data Spaces}

\begin{definition}[Euclidean cones]
	\label{Def:Cones}

	The \emph{Euclidean cone} over a metric space $(N,d_N)$ is the quotient
	$$C_0 N = N \times [0, \infty) / \sim$$ with the equivalence relation
	$(\eta,s) \sim (\theta,t)$ for $\theta, \eta \in N$, $s,t \geq 0$, if and
	only if $(\eta,s) = (\theta,t)$ or $s = 0 =t$. We denote the class of
	$(\eta,0)$ by $\mathcal{O}$ for $\eta \in N$ and refer to it as the
	\emph{cone point}.

	The cone distance $d_0$ between any two points $x = (\eta, s) \in C_0 N$
	and $y = (\theta, t) \in C_0 N$ is given by
	\begin{align*}
		d_0(x,y) = \Big( s^2 + t^2 -2 s t \cos\big(d_\pi(\eta,\theta) \big)
		\Big)^{\frac{1}{2}},
	\end{align*}
	where $d_\pi(\eta, \theta) = \min\{d_N(\eta, \theta), \pi\}$.
\end{definition}

A trivial example is a Euclidean space viewed as the Euclidean cone over a
unit hypersphere. In general, $(C_0N,d_0)$ is a metric space as the proposition below
teaches and the cone is called \emph{Euclidean} since for $(\eta, s),(\theta,t)
	\in C_0N$ with $d_N(\eta, \theta)\leq \pi$ their $d_0$ distance is Euclidean,
if $(\eta, s),(\theta,t)$ are viewed as polar coordinates.

\begin{proposition}

	Let $(N,d)_N$ be a metric space. The Euclidean cone equipped with the distance $d_0$ from
	Definition \ref{Def:Cones} yields a metric space $(C_0 N,d_0)$.

\end{proposition}

\begin{proof}
	Proposition I.5.9 and II.3.14 in \cite{BH11}.
\end{proof}

\begin{example}[BHV tree spaces \cite{BHV02}]
	\label{ex:BHV}
	The \emph{Billera-Holmes-Vogtmann tree spaces} of phylogenetic trees can
	be constructed as the cone over certain simplicial complexes, referred to
	as \emph{links}, of which the faces correspond to different tree topologies
	(Section 1 and 4.1 in \cite*{BHV02}). The link of trees with four leaves
	is the well known \emph{Peterson graph} and the resulting tree space $\mathcal{T}_4$
	is then Euclidean cone over the Peterson graph, see Figure \ref{fig:BHVlink}. There, in the left panel, each edge is
	assigned length $\pi / 2$ and corresponds to a binary tree topology, whereas the
	nodes correspond to trees with a single interior edge.
	\begin{figure}[!h]
		\centering
		\subfloat[{\it The Peterson graph.}]{
			\includegraphics*[width=5cm, height=5cm]{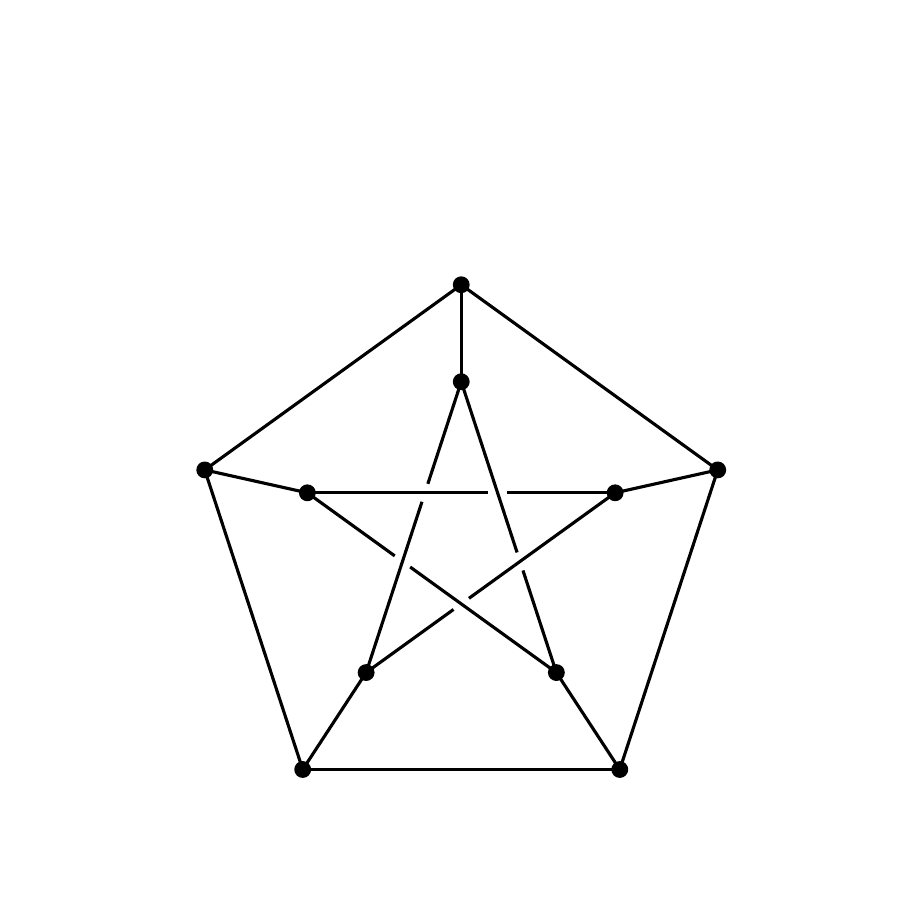}}
		\subfloat[{\it Three maximal orthants of $\mathcal{T}_4$ (gray) with corresponding part of the  link (black)}]{
			\includegraphics*[width=4cm, height=4cm]{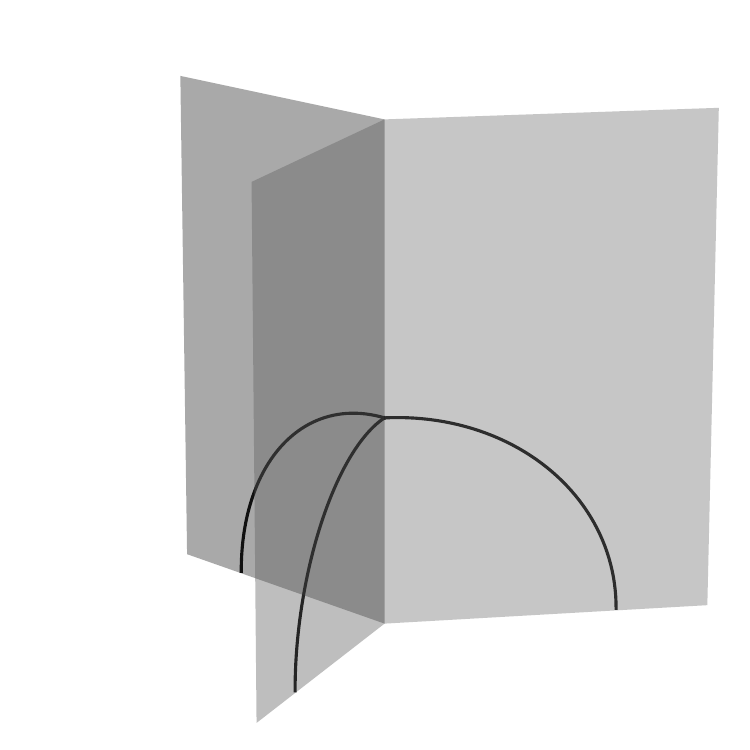}}
		\caption{\it The BHV tree space $\mathcal{T}_4$ is the Euclidean cone over
			the Peterson graph, where each edge has length $\pi/2$.}
		\label{fig:BHVlink}
	\end{figure}
\end{example}

\begin{example}[Stickiness on the $K$-spider]
	\label{ex:kspider}

	Consider the space obtained by gluing together $K \geq 3$
	half-lines at their endpoints, i.e. the Euclidean cone
	$$\mathfrak{A}_K := C_0 N = \{
		1,2,\ldots K\} \times [0,\infty)/\sim $$
	over the space $N= \{
		1,2,\ldots K\}$ with the trivial metric $d_N(i,i) = 0$ and $d_N(i,j) =
		\pi$ for $1 \leq i \neq j \leq K$, yielding the  cone's metric
	\begin{align*}
		d_0(i_1,t_1), (i_2,t_2)) =
		\begin{cases}
			\lvert t_1 - t_2\rvert & \quad \text{if } i_1 = i_2, \\
			t_1 + t_2              & \quad \text{else.}          \\
		\end{cases}
	\end{align*}
	This space is also called a $K$-spider and in case of $K=3$ it coincides with the BHV space  $\cT_3$, cf. Example \ref{ex:BHV}.

	Figure \ref{fig:kspidersticky} shows an example of stickiness on the
	3-spider: The Fr\'echet mean of three equally weighted point masses at same distance to
	the cone point $\cO$ \emph{sticks} to $\cO$, if one of the point masses is perturbed.

	\begin{figure}[!h]
		\begin{subfigure}{0.5\textwidth}
			\includegraphics[scale=.7]{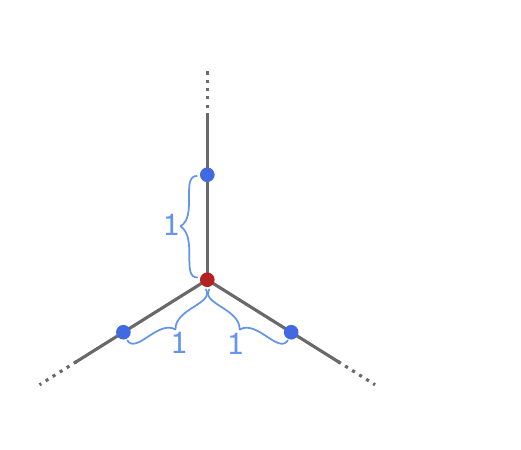}
		\end{subfigure}
		\begin{subfigure}{0.5\textwidth}
			\includegraphics[scale=.7]{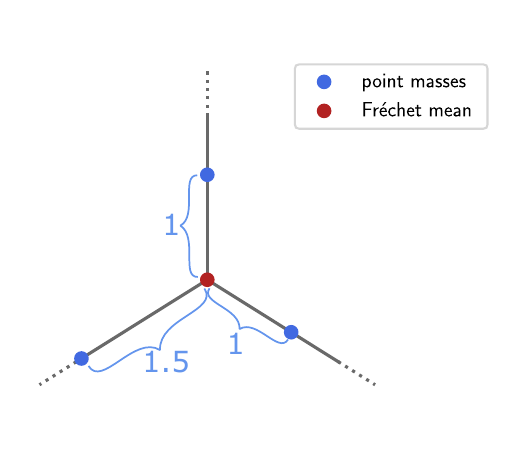}
		\end{subfigure}
		\caption{\it Even though one of three equally weighted point masses from the left panel is perturbed in the right panel, their Fr\'echet mean
			remains at the cone point.} 
		\label{fig:kspidersticky}
	\end{figure}
\end{example}

The $K$-spider is an example of the larger class of spaces called
\emph{open books} for which stickiness has been investigated in
\cite{HHL+13}. The $K$-spiders form the open books of dimension 1.

\begin{example}[Stickiness on the open books]
	\label{ex:open-books}
	With the notation of the previous example, an open book with $K$ pages of dimension $d \geq 2$ is given by
	$\mathfrak{B}_K^d = \mathfrak{A}_K \times \RR^{d-1}$ equipped with the product metric. Indeed, for $d\geq 2$,
	distributions, constructed similarly as above for the 3-spider have
	Fr\'echet means sticking to the $d-1$-dimensional subset $S:=\{\mathcal{O}\}
		\times R^{d-1}$, called the \emph{spine}.
	For these open books consider the \emph{folding maps}
	\begin{align*}
		\phi_j : & \mathfrak{B}_K^d \to \RR              \\
		         & (k, x_0, x_1, \ldots x_{d-1}) \mapsto
		\begin{cases}
			x_0  & \quad \text{if } j=k \;; \\
			-x_0 & \quad \text{else, }
		\end{cases}
	\end{align*}
	for $j \in \{1, \ldots,K\}$.

	Then, given $\prb \in \wst{\mathfrak{B}_K^d}$, its \emph{$j$-th folded
		moment}, $j \in \{1, \ldots,K\}$, is
	\begin{align}
		\label{eq:foldbook}
		m_j(\prb) = \int_{B_K^d} \phi (x) \ \diff \prb(x).
	\end{align}

	Now suppose $m_j(\prb) < 0$ for all $j\in \{1, \ldots,K\}$ and that $X_1, X_2, \ldots \iid \prb$ is a sample. Then it was shown \cite[Theorem 4.3]{HHL+13} that the sample Fr\'echet mean sticks to the set $S=\{ \mathcal{O}
		\} \times \RR^{d-1}$. i.e. that there is
	random $N \in \NN$ such that for all $n \geq N$
	\begin{align*}
		\argmin_{x \in B_K^d} \sum_{i=1}^n d^2(x, X_i) \subseteq \{ \mathcal{O} \} \times \RR^{d-1} \mbox{ a.s.}
	\end{align*}

	%

\end{example}


\subsection{Basic Sticky Flavors in Scenario I}

\begin{definition}[Three basic sticky flavors \ref{set:1}]
	\label{def:sticky}

	Let $\mathcal{Q} \subseteq \wst{M}$ be equipped with a topology $\tau$ and $A\subseteq N$ be measurable.
	Then, with respect to a closed $S\subset M$, a probability distribution $\prb \in
		\cQ$ is called
	\begin{description}
		\item[$(\cQ,\tau)$ topologically sticky] (T) if there is $U \in \tau$ such that $\prb \in U$ and $\mathfrak{b}(\prbalt)
			      \subseteq S$ for all $\prbalt \in U$,
		\item[$A$-perturbation sticky] (P), if for every $y \in A$
		      there is $t_y > 0$ such that
		      $\mathfrak{b}(\prb_y^t) \subseteq S$ for the \emph{perturbed measures} $\prb_y^t:=(1-t) \cdot \prb + t \cdot \delta_y$,
		      for all $0 < t < t_y$,
		\item[sample sticky] (S) if for every sample $(X_i)_{i \in \mathbb{N}}\iid\prb$, there is a random $N\in \NN$ such that $\mathfrak{b}(\prb_n) \subseteq S$ for all $n \geq N$ a.s.
	\end{description}
\end{definition}

As the following example teaches, under no additional assumptions, topological stickiness does not imply the other two flavors.

\begin{example}[(T) $\not\Rightarrow$ (P), (T)]

	Consider $A= \RR$ and $\cQ=\wst{\RR}$ equipped with the discrete topology $\tau$. Then
	trivially, every probability distribution is $\tau$ sticky. Since the
	Fr\'echet mean on $\RR$ is just the expectation, there can be no 	$\RR$-perturbation sticky distribution. Indeed, for $\prb \in \wst{\RR}$, we
	have $\mathfrak{b}(\prb) = \{\mean\}$, for some $\mu \in \RR$, and for any
	$y \in \RR \setminus \{ \mean\}$ and $t \in (0,1]$, we have for the
	perturbed measure $\prb^y_t = (1-t) \prb + t \delta_y$
	\begin{align*}
		\mathfrak{b}(\prb^y_t) & = \int_M x \ \diff \prb^y_t(x)                \\
		                       & = (1-t) \cdot \mean + t \cdot y \neq \mean\,.
	\end{align*}
	Moreover, the only sample sticky distributions in $\wst{\RR}$ are Dirac measures.
\end{example}

Recalling that for the empirical measures $\prb_n$ of $\prb \in \wst[1]{M}$, by definition and the classical law of large numbers ($\prb_n \to \prb$ a.s.) we have $\prb_n \to \prb$ in the \emph{weak
	topology} of $\cP(M)$, i.e.
\begin{align*}
	\lim_{n \to \infty} \int_M f(z) \ \diff \prb_n(z) = \int_M f(z) \ \diff \prb(z)
\end{align*}
for any bounded and continuous $f: M \to \mathbb{R}$, we have at once the following.

\begin{theorem}
	[($\prb_n\to \prb$ a.s. in T) $\Rightarrow$ (S)]\label{thm:convergence-sticky}

	Let $\tau$ be a topology for some $\mathcal{Q}\subseteq \wst{M}$ for which
	$\prb_n \to \prb$ a.s. for every $\prb \in \mathcal{Q}$. Then every topologically
	sticky measure $\prb\in \wst{M}$ is also sample sticky.

	In particular, if $\tau$ is the weak topology for $\wst{M}$ then topologically sticky $\prb \in \wst{M}$ is also sample sticky.
\end{theorem}

Typical topologies of high interest are derived from Wasserstein distances and from $f$-divergences.  As we will see, they come with different flavors of topological stickiness. To guarantee the existence of optimal couplings  \cite[Theorem 4.1]{V08}, for the Wasserstein distance, we require the following scenario of assumptions.

\begin{setting}
	\label{set:1}
	Assume that $(M,d)$ is \emph{complete} and \emph{separable} metric space.
\end{setting}


\begin{definition}
	\label{def:wstein}

	In Scenario I, for $\prb, \prbalt
		\in \wst{M}$, let $\Pi(\prb, \prbalt)$ denote the set of all
	couplings of the two distributions. Then for $q \geq 1$, the \emph{$q$-th
		Wasserstein distance} is given by
	\begin{equation*}
		W_q(\prb, \prbalt) = \left(\inf_{\varpi \in \Pi(\prb,\prbalt)}
		\int_{M \times M} d^q(x,y) \ \diff \varpi(x,y)\right)^\frac{1}{q}, \quad \prb, \prbalt \in \wst[q]{M}.
	\end{equation*}
\end{definition}

\begin{definition}
	\label{def:divergence}

	For a convex function  $f: [0, \infty) \to (-\infty, \infty]$ satisfying
	\begin{eqnarray}
		\label{eq:convex-fcn}
		\lim_{x \to 1} f(x)= f(1) = 0&\mbox{ and }&\lim_{x \to 0} f(x) = f(0)< \infty\,,
	\end{eqnarray}
	the \emph{$f$-divergence} $D_f$, of $\prb, \prbalt \in \cP(M)$ is given by 	\begin{align*}
		D_f(\prb \Vert \prbalt) = \int_M f\left(\frac{p(x)}{q(x)}\right) q(x) \ \diff \lambda(x),
	\end{align*}
	with the Radon-Nikodym derivatives
	$p = \frac{\diff \prb}{\diff\lambda}, q =
		\frac{\diff \prbalt}{\diff\lambda}$ with respect to a dominating measure $\lambda$, e.g. $\lambda = \frac{1}{2}(\prb + \prbalt)$, see \cite{fdiver}.
\end{definition}

Some examples of popular $f$-divergences are displayed in Table \ref{tab:fdiver}.

\begin{table}
	\centering
	\fbox{
		\begin{tabular}{ c | c }
			$f$-divergence             & $f(x)$                                                            \\
			\hline
			Kullback-Leibler           & $x \cdot \log(x)$                                                 \\
			total variation  distance  & $\frac{1}{2} \lvert x - 1 \rvert$                                 \\
			Jensen-Shannon             & $-(x+1) \cdot \log \left( \frac{x+1}{2}\right)+ x \cdot \log (x)$ \\
			squared Hellinger distance & $2 \cdot (1 - \sqrt{x})$                                          \\
		\end{tabular}}
	\caption{\it Some fashionable $f$-divergences satisfying the requirements of Definition \ref{def:divergence}, see e.g. \cite{fdiverbounds}.}
	\label{tab:fdiver}
\end{table}

\begin{definition}[Flavors of topological stickiness in Scenario \ref{set:1}]

	In Scenario I, for $q\geq 1$ and $\mathcal{Q} \subset \wst[q]{M}$ consider the Wasserstein distance $W_q$ and let $D_f$ be an
	$f$-divergence. Then, with respect to a closed set $S\subset M$ we say
	that $\prb \in \mathcal{Q}$ is

	\begin{description}
		\item[Wasserstein-$q$ sticky] $(\cQ,W_q)$, if there is $\epsilon > 0$  such that $\mathfrak{b}(\prbalt)\subseteq S$ for all $\prb \in \cQ$ with $W_q(\prb, \prbalt) < \epsilon$.
		\item[$f$-divergence sticky] $(\cQ,D_f)$, if there is $\epsilon > 0$ such that $\mathfrak{b}(\prbalt)\subseteq S$ for all $\prb \in \cQ$ with $D_f(\prb \Vert \prbalt) < \epsilon$.
	\end{description}
	In particular for $(D_f)$ with $f(x) =\frac{1}{2} \lvert x - 1 \rvert $, we speak of \emph{total variation} ($\cQ$, TV) \emph{sticky}, respectively, see Table \ref{tab:fdiver}.
\end{definition}

In Scenario I, we relate the various flavors as follows. In particular,
under mild conditions, Wasserstein stickiness implies the two other flavors.
Later, under more restrictive scenarios warranting additional machinery, Theorem
\ref{theorem:mainthm} elucidates when the converse holds.

\begin{theorem} 
	\label{theorem:immediateflvrs}
	In Scenario I, let $S \subset M$ be closed and for measurable $A \subset M$, $q \geq 1$, define
	\begin{align*}
		\mathcal{Q}^q_A & = \{ P \in \wst[q]{M} \ \vert \ \supp(P) \subseteq A \}\,. \\
	\end{align*}

	Then the following hold for $\prb \in \mathcal{Q}^q_{A}$ for stickiness with respect to $S$:
	\begin{itemize}
		\item[(i)]  If $\prb$ is $(\mathcal{Q}^1_A, W_1)$ sticky and $A$ is bounded, then $\prb$ is $(\mathcal{Q}^1_A, \TV)$ sticky,
		\item[(ii)]  If $\prb$ is $(\mathcal{Q}^q_A, \TV)$ sticky, $f$ is
		      twice-differentiable and $f^{\prime \prime}(x) > 0$ for all $x \neq 1$, then it is $(\mathcal{Q}^q_A, D_f)$ sticky.
		\item[(iii)] If $\prb$ is $(\mathcal{Q}^q_A, D_f)$-sticky, then it is $A$-perturbation sticky.
		\item[(iv)] If $\prb$ is $(\mathcal{Q}_A^1, W_1)$-sticky then it is $(\mathcal{Q}_A^q, W_q)$ sticky.
		\item[(v)] If $\prb$ is $(\mathcal{Q}^q_A, W_q)$-sticky then it is $A$-perturbation and sample sticky.

	\end{itemize}
\end{theorem}

\begin{proof}

	(i) follows at once from 
	\begin{align*}
		W_1(\prb,\prbalt) \leq \diam(A) \cdot \TV(\prb, \prbalt),
	\end{align*}
	where $\diam(A) = \sup_{x,y \in M} d(x,y)$ denotes the diameter, cf.
	\cite[Theorem 6.15]{V08}.

	(ii) for such $f$, there is a constant $C_f$ such that
	\begin{align*}
		\TV(\prb, \prbalt) \leq C_f \cdot D_f(\prb \lVert \prbalt)
	\end{align*}
	for all $\prb,\prbalt \in \cP(M)$, see \cite{gilardoni}. 

	By assumption in (iii), there is $\epsilon > 0$ such that
	$\mathfrak{b}(\prbalt) \subseteq S$ for any $\prbalt \in \cQ_A^q$ with
	$D_f(P \Vert Q) < \epsilon$. Hence, for any point $y \in A$ and $t \in
		[0,1)$, giving rise to the perturbed measure $\prb_y^t = (1-t) \prb + t
		\delta_y \in \cQ_A^q$, we have $\prb\ll \prb_y^t \in \wst{M}$, and with $w_y =
		\prb(\{y\}) \geq 0$ thus
	\begin{equation}
		\label{eq:fdvirpert}
		\begin{split}
			D_f(\prb \Vert \prb_y^t) & =
			\int_{M \setminus \{y\}} f\left(\frac{1}{1-t}\right) \ \diff \prb_y^t(z)
			+ \int_{\{y\}} f\left( \frac{w_y}{(1-t)w_y + t}\right) \ \diff \prb_y^t(z)                                                                    \\
			                         & = (1-t)	(1- w_y) f\left( \frac{1}{1- t}\right) + \big(t + (1-t)w_y\big) f\left( \frac{w_y}{(1 - t)w_y + t}\right).
		\end{split}
	\end{equation}
	By assumptions in (\ref{eq:convex-fcn}) on $f$, observe that $D_f(\prb \Vert \prb_y^t) \to 0$ for $t \to 0$. Thus, $D_f(\prb \Vert \prb_y^t) < \epsilon$ for small enough $t$, and hence $\mathfrak{b}(\prb_y^t) \subseteq S$ as asserted in (i).

	(iv) This is an immediate consequence of the fact that
	\begin{align*}
		W_1(\prb, \prbalt) \leq W_q(\prb, \prbalt) \quad \prb, \prbalt \in \wst[q]{M}\,,
	\end{align*}
	see \cite[Remark 6.6]{V08}.\\
	(v): We first establish $A$-perturbation stickiness. Once again taking any point $y \in A$ and $t \in [0,t]$, giving rise to the perturbed measure $\prb_y ^t = (1-t)
		\cdot \prb + t \cdot \delta_y$, pick optimal couplings $\varpi_0 \in \Pi(\prb,\prb)$ and $\varpi_1 \in \Pi(\prb, \delta_y)$ with respect to the cost function $d^q$. Then, $\varpi_t = (1-t) \cdot \varpi_0 + t \cdot
		\varpi_1 \in \Pi(\prb, \prb_y^t)$. Hence,
	\begin{align}\label{eq:Wasserstein-perturbtion}
		W^q_q(\prb, \prb_y^t) \leq t \cdot W^q_q(\prb, \delta_y)\,,
	\end{align}
	yielding $\mathfrak{b}(\prb_y^t) \subseteq S$  for small enough
	$t$ by hypothesis on $W_q$ stickiness as, by the above, $\prb_y^t \to \prb$ for $t \to 0$ in $W_q$ distance.

	The second assertion follows at once from Theorem
	\ref{thm:convergence-sticky} since empirical distributions $\prb_n$
	converges almost surely to $\prb \in \wst[q]{M}$ in $W_q$ distance, due to
	\cite[Theorem 6.9]{V08}. Here, note that almost surely $\supp(\prb_n) \subseteq \supp(\prb)$.
\end{proof}

\subsection{Counterexamples}

In general, topologies generated by $f$-divergences do not satisfy a.s.
convergence of empirical measures as required in Theorem
\ref{thm:convergence-sticky}.

\begin{example}[$P_n\not\to P$ a.s. in TV]
	For the total variation distance,
	\begin{align*}
		\TV(\prb, \prbalt) = \sup_{B \in \mathcal{B}(M)} \lvert \prb(B) - \prbalt(B)\rvert,
	\end{align*}
	(see \cite[Definition 2.4]{tsy}) where $\mathcal{B}(M)$ denotes the Borel $\sigma$-algebra of $(M,d)$, $\prb,\prbalt \in \cP(M)$, the empirical distribution $\prb_n$ does converge a.s. to its population distribution $\prb$, if the latter is continuous, since then, a.s.,
	$$\TV(\prb, \prb_n) = 1\,.$$
\end{example}

\begin{example}\label{ex:kale}[The kale \cite{HMMN15}, equivalence of flavors]
	\label{ex:kale}

	Consider the set $\mathbb{R} / \alpha \mathbb{Z}$ for $\alpha > 2
		\pi$. We equip this set with the metric
	\begin{align*}
		d_\alpha(\theta_1, \theta_2) =
		\min_{n \in \mathbb{Z}}\lvert \theta_1 - \theta_2 + n \alpha \rvert
		,\quad \theta_1, \theta_2 \in \mathbb{R} / \alpha \mathbb{Z}.
	\end{align*}
	One can interpret this as a circle with angle sum $\alpha$ larger than
	$2\pi$. \\
	The Euclidean cone over this set, called \emph{the kale}
	$\mathcal{K}_\alpha$, has been of special interest as probability
	measures can exhibit \emph{stickiness} at the cone point \cite*{HHL+13}.
	An example of geodesics are shown in Figure \ref{fig:kalegeod}.\\
	Similar to open books, it was observed that a distribution $\prb \in
		\wst{\mathcal{K}_\alpha}$ is sample sticky at the cone point if for all
	$\theta \in [0,\alpha)$, the \emph{first folded moment} $m_\theta$ is
	negative, i.e.
	\begin{align*}
		0 > m_\theta = \int_\prb r \cdot \cos\left(\min\{\pi, d_\alpha(\theta, \tilde{\theta}\}\right) \ \diff (r, \tilde{\theta}).
	\end{align*}
	Even more, the equivalence of this condition to $\mathcal{K}_\alpha$-perturbation stickiness
	and $(\wst{\mathcal{K}_\alpha},W_1)$-stickiness was shown in \cite[Theorem 7.6]{HMMN15}.
	\begin{figure}[!h]
		\centering
		\includegraphics[scale=1]{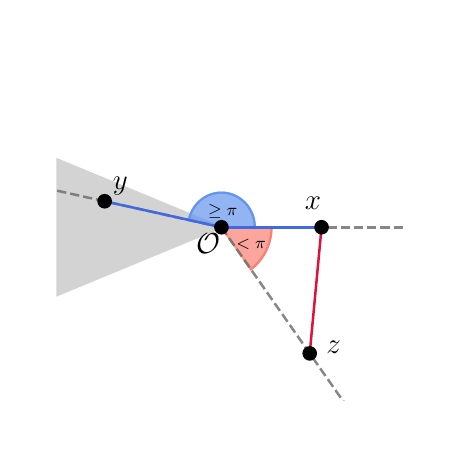}
		\caption{\it As the angle between $x$ and $y$ at $\cO$ is at least $\pi$, the geodesic path between them passes through the cone point $\cO$. We say the direction from $\cO$ to $y$ lies in the "shadow" of the direction from $\cO$ to $x$.
			However, the geodesic between $x$ and $z$ does not, as the angle between them at $\cO$ is less than $\pi$.
		}
		\label{fig:kalegeod}
	\end{figure}
\end{example}

\begin{example}[(S) $\not\Rightarrow$ (W), (P)]
	\label{ex:samplenotW1}

	Consider the two-dimensional open book with three pages $M=\mathfrak{B}_3^2$,
	cf. Example \ref{ex:kspider}. As a direct product of
	Hadamard spaces, $M$ is again a Hadamard space and for any $P \in \wst{M}$,
	its unique mean is given by $\mathfrak{b}(\prb) = (\mathfrak{b}(\prb'),
		\mathfrak{b}(\prb''))$, where $\prb'$ and $\prb''$ are the marginal
	distributions on $\mathfrak{A}_3$ and $\RR$, respectively, see \cite[Lemma
		3.9 and Proposition 5.5]{Sturm}. Then consider $$\prb=
		\frac{1}{3}\sum_{j=1}^3 \delta_{((j,1),0)}\,,$$ the sum of three point
	masses with weight $1/3$, distance 1 to the origin at height $0 \in \RR$,
	see Figure \ref{fig:samplenotW1}. Clearly, a sample Fr\'echet
	mean of i.i.d. samples of that distribution is almost surely at height 0.
	Furthermore, one may verify by computing the first folded means, see
	Equation (\ref{eq:foldbook}), that the distribution is sample sticky
	on the spine $S=\{\cO\} \times \RR$. Combining these two observations, we see
	that the distribution is actually sample sticky on the single point
	$S=\{(\mathcal{O}, 0)\}$. \\
	However, as displayed in Figure \ref{fig:samplenotW1}, the distribution is
	not $\mathfrak{B}_3^2$-perturbation sticky and, consequently, not
	$(\wst{\mathfrak{B}_3^2}, W_1)$-sticky by Theorem \ref{theorem:immediateflvrs}.
	Perturbations with point masses at different heights will always shift the
	Fr\'echet mean along the spine $\{ \cO \} \times \RR$.



	\begin{figure}[!h]

		\centering
		\includegraphics[scale=.8]{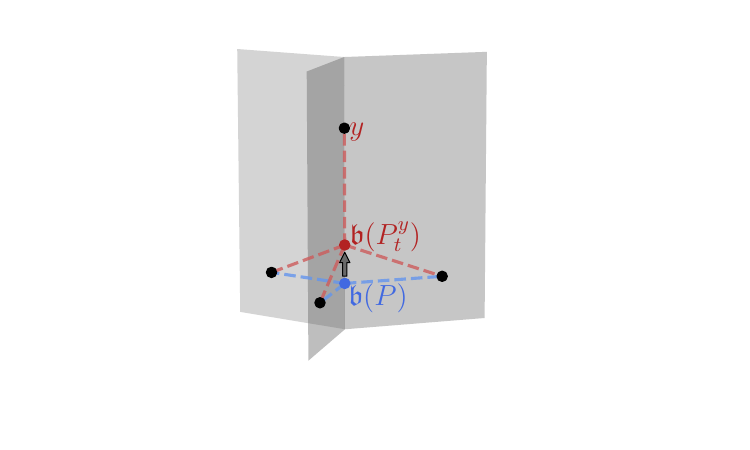}
		\caption{\it The distribution $\prb$ is sample sticky but neither Wasserstein nor perturbation sticky.}
		\label{fig:samplenotW1}
	\end{figure}

\end{example}


\begin{example}[P-sticky $\not\Rightarrow$ W-sticky]
	\label{ex:pertnotW1}

	Consider the set $M = \{0 \} \cup \{ 1/n \ \vert \ n \in \NN\}$ equipped
	with the following metric
	\begin{align*}
		d(x,y) =
		\begin{cases}
			0     & \text{if } x = y, \\
			x + y & \text{else.}
		\end{cases}
		\
	\end{align*}
	It is easy to see that $(M,d)$ is a metric space that is separable (because
	it is countable), bounded and complete (all non-trivial Cauchy-sequences
	converge to $0$).
	Every probability measure on $M$ has arbitrary moments, and every
	Wasserstein distance is well-defined.

	Consider now $\prb = \delta_0$. For every $\epsilon >0$ we can find an $n
		\in \NN$ such that $\frac{1}{n} < \epsilon$. Hence, $W_1(\delta_0,
		\delta_{1/n}) = d(0, 1/n) = 1/n < \epsilon$. Since
	$\mathfrak{b}(\delta_{1/n}) =\{ 1/n\}$, as it is just point mass,
	$\delta_0$ is not $(M,W_1)$-sticky.

	On the other hand, for every $1/n = x\in M$ and all $0\leq t \leq 1$, we
	have that $\mathfrak{b}((1-t) \delta_0 + t \delta_x) = \{0\}$ as
	\begin{align*}
		2 \cdot F_{(1-t) \delta_0 + t \delta_x} (0) = t \cdot \frac{1}{n^2}
		 & < (1-t) \cdot \frac{1}{m^2} + t \left(\frac{1}{m}+\frac{1}{n}\right)^2 \\
		 & =
		(1-t) \cdot d(0,y)^2 + t d(x,y)^2                                         \\
		 & = 2 \cdot F_{(1-t) \delta_0 + t \delta_b}(x),
	\end{align*}
	for any $1/m=y \in M$. Therefore, $\delta_0$ is $M$-perturbation sticky.

\end{example}


%
%

\section{CAT($\kappa$) Spaces and Spaces of Their Directions}\label{scn:CAT}


The spaces mentioned in the introduction and in the examples of Section \ref{scn:3-flavors} have in common that they are \emph{spaces of global nonpositive curvature} or \emph{CAT(0) spaces}. In this section, we will first review some basic metric geometry for \cat~spaces and then introduce their tangent cones and spaces of directions. This will be used for studying stickiness in the subsequent sections. The following definitions and results can be found in \cite{BB08},\cite{BH11}, \cite{Sturm}, \cite{LN18} and
\cite{BBI01}.

\subsection{A Brief Review of CAT($\kappa$) Spaces}
\begin{definition}[Geodesics and geodesic spaces]

	For two distinct points $x, y \in M$, a \emph{curve} $\gamma$ from $x$ to
	$y$ is a continuous map
	\begin{align*}
		\gamma : [a,b] \to M,\quad a,b \in \mathbb{R}\,,
	\end{align*}
	with $\gamma(a) = x$ and $\gamma(b) = y$, having length
	\begin{align*}
		L(\gamma) = \sup_{\substack{ a \leq t_1 < \ldots < t_n \leq b \\
				n \in \mathbb{N}\setminus \{1\}
			}} \sum_{i=1}^{n-1} d(\gamma(t_{i}), \gamma(t_{i+1})).
	\end{align*}

	If $L(\gamma|_{[a,t]}) = d(x,\gamma(t))$ for all $a<t\leq b$ where $\gamma|_{[a,t]}$ is the curve $\gamma$ restricted to $[a,t]$ we say that $\gamma$ is a \emph{geodesic}. In fact, then it is a \emph{minimizing unit speed geodesic}, but as we are not concerned with geodesics of other speeds, we refer to them simply as geodesics.

	The metric space $(M,d)$ is called a \emph{geodesic space} if any two
	distinct points can be joined by a geodesic.

	For $r>0$ the metric space $(M,d)$ is called \emph{$r$-geodesic} if any 	two distinct points $x,y\in M$ with $d(x,y) < r$ can be joined by a geodesic. If the space is $r$-geodesic for all $r>0$ it is called a \emph{geodesic space}.

\end{definition}

Notably the space $N=\{1,\ldots,K\}$ with the trivial metric, whose Euclidean
cone yielded the $K$-spider, cf. Example
\ref{ex:kspider}, is $\pi$-geodesic but not geodesic.

The notion of convex sets and functions in vector spaces generalizes naturally
to geodesic metric spaces.

\begin{definition}[Convex sets and functions]

	Let $(M,d)$ be a geodesic space. Then $A \subseteq M$ is called \emph{convex} if for any two points, $x,y \in M$, all geodesics joining 	$x$ and $y$ are contained in $A$.

	A function $f : M \to \RR$ is called \emph{convex} if all concatenations 	$f \circ \gamma$ with geodesics $\gamma: [a,b] \to M$ are convex, i.e. for any $t \in [0,1]$ and $c,d \in [a,b]$,
	\[
		f\big(\gamma((1-t) \cdot c + t\cdot d) \big) \leq (1-t) \cdot f\big(\gamma(c)\big)
		+ t \cdot f\big(\gamma(d)\big).
	\]

\end{definition}

For $\kappa \in \RR$, we define the \emph{model} or \emph{reference spaces} $(M_\kappa, d_\kappa)$ as the metric spaces induced by
the two-dimensional Riemannian manifold structures with constant sectional curvatures $\kappa$ as follows.

\begin{definition}[Reference spaces]

	Let
	\begin{align*}
		M_\kappa :=
		\begin{cases}
			\mathbb{H}^2 \quad & \text{if } \kappa < 0, \\
			\RR^2 \quad        & \text{if } \kappa = 0, \\
			\mathbb{S}^2 \quad & \text{if } \kappa > 0, \\
		\end{cases}
	\end{align*}
	where
	\begin{align*}
		\mathbb{S}^2 & := \{ (x_0, x_1, x_2) \in \RR^3 \ : \ x_0^2 + x_1^2 + x_2^2 =1 \},              \\
		\mathbb{H}^2 & := \{ (x_0, x_1, x_2) \in \RR^3 \ : \ x_0^2 + x_1^2 - x_2^2 = -1, \ x_2 > 0 \},
	\end{align*}
	are the two-dimensional unit sphere and hyperbolic plane, respectively, with hyperbolic and spherical distances, respectively, given by
	\begin{align*}
		d_{-1}(x,y) & := \arcosh (- x_0 y_0 - x_1 y_1 + x_2 y_2 ) & \mbox{ for } x, y\in \mathbb{H}^2    \\
		d_1(x,y)    & := \arccos (x_0 y_0 + x_1 y_1 + x_2 y_2 )   & \mbox{ for } x, y\in \mathbb{S}^2\,,
	\end{align*}
	for $x = (x_0,x_1,x_2), y = (y_0,y_1,y_2)$. Then, obtain $d_\kappa$ by rescaling
	\begin{align*}
		d_\kappa(x, y) :=
		\begin{cases}
			\frac{1}{\sqrt{- \kappa}} d_{-1}(x, y) \quad & \text{if } \kappa < 0, \\
			\lVert x - y \rVert_2 \quad                  & \text{if } \kappa = 0, \\
			\frac{1}{\sqrt{\kappa}} d_{1}(x, y) \quad    & \text{if } \kappa > 0, \\
		\end{cases}
	\end{align*}
	for $x,y \in M_\kappa$,
	where $\lVert \cdot \rVert_2$ is the Euclidean norm. Furthermore, we
	define the \emph{geodesic half radii}
	\begin{align}\label{eq:R_kappa}
		R_\kappa :=
		\begin{cases}
			\infty \quad                    & \textrm{if } \kappa \leq 0, \\
			\frac{\pi}{\sqrt{\kappa}} \quad & \textrm{else.}
		\end{cases}
	\end{align}
\end{definition}

\begin{definition}[Reference triangles]
	Let $x_1,x_2,x_3 \in  M$ be distinct points and $\kappa \in \RR$. Then  $x'_1,x'_2,x'_3 \in  M_\kappa$ form a \emph{reference triangle} in $(M_\kappa,d_\kappa)$ if
	\begin{align*}
		d(x_i,x_j) = d_\kappa(x_i^\prime, x_j^\prime) \quad \forall i,j \in \{1,2,3\}.
	\end{align*}
\end{definition}

The following lemma teaches that for triangles of sufficiently small
perimeter, reference triangles in $M_\kappa$ always exist.

\begin{lemma}
	\label{lem:reference-triangle}
	Let $\kappa \in \RR$. Then for any three distinct points
	$x_1,x_2,x_3 \in M$ with perimeter $d(x_1,x_2) + d(x_1,x_3) + d(x_2,x_3) \leq 2 R_\kappa$, there exists a reference triangle in $(M_\kappa, d_\kappa)$.
\end{lemma}

\begin{proof}
	\cite[Lemma I.2.14]{BH11}
\end{proof}

\begin{definition}[CAT($\kappa$) space]
	\label{Def:CATk}
	Let  $\kappa \in \RR$. Then an $R_\kappa$-geodesic space $(M,d)$ is a CAT($\kappa$) space if for any distinct $x_1,x_2,x_3 \in M$ with perimeter $d(x_1,x_2) + d(x_1,x_3) + d(x_2,x_3) \leq 2 R_\kappa$ and reference triangle $x'_1,x'_2,x'_3 \in M_\kappa$,
	\begin{align*}
		d(x_1, \gamma(t)) \leq d_\kappa(x'_1, \gamma'(t)), \mbox{ for all }0\leq t \leq d(x_2,x_3)\,,
	\end{align*}
	where $\gamma$ is a geodesic in $(M,d)$ from $\gamma(0)=x_2$ to $\gamma(d(x_2,x_3))= x_3$ and $\gamma'$ is a geodesic in $(M_\kappa,d_\kappa)$ from $\gamma'(0)=x'_2$ to $\gamma'(d(x_2,x_3))= x'_3$, see Figure \ref{fig:npc-quad}.

	A complete CAT($0$) is also called a \emph{Hadamard space}.
\end{definition}

Here ``CAT'' stands for the three mathematicians Cartan, Alexandrov and
Topogonov, and $\kappa$ bounds the global curvature from above, see
Proposition \ref{prop:catgeod}. In contrast, while local curvature can be zero
as is the case for circles, cylinders and tori, the unit circle is only
CAT($\kappa$) for all $\kappa \geq 1$. The following proposition collects some
well-known results for \cat ~spaces.

\begin{proposition}
	\label{prop:catgeod}

	Let $(M,d)$ be a \cat ~space for $\kappa \in \RR$. Then, the following
	hold:
	\begin{enumerate}
		\item[(i)] any two distinct $x,y\in M$ with $d(x,y) < R_\kappa$ can be joined by a unique geodesic,
		\item[(ii)] $(M,d)$ is a CAT($\kappa^\prime$) space for all $\kappa^\prime \geq \kappa$.
		\item[(iii)] if $\kappa > 0$, then $(M,\breve{d})$ is a CAT(1)-space, where
		      $\breve{d}(x,y): = \sqrt{\kappa} \cdot d(x,y)$ for $x,y \in M$.
		\item[(iv)] take three points $x,y,z \in M$, and if $\kappa >0$ assume that $x,y,
			      z \in U \subset M$ with $\diam (U) \leq R_\kappa / 2 - \epsilon$
		      for $ 0 < \epsilon < R_\kappa /2$, then
		      \begin{align}
			      \label{ineq:NPC}
			      d^2(z, \gamma_x^y(t \cdot l)) \leq (1-t) \cdot d^2(z, x) + t \cdot d^2(z, y) - C_{\kappa,\epsilon} (1-t) \cdot t \cdot d^2(x, y),
		      \end{align}
		      for all $t \in [0,1]$, where $\gamma_x^y: [0,l] \to U$ is the unique  (see (i)) geodesic from $x$ to $y$ with $l = d(x,y)$ and
		      \begin{align*}
			      C{\kappa,\epsilon} =
			      \begin{cases}
				      1 \quad                                                               & \text{if } \kappa \leq 0, \\
				      (\pi - 2\epsilon\sqrt{\kappa})\cdot \tan(\epsilon\sqrt{\kappa}) \quad & \text{else.}
			      \end{cases}
		      \end{align*}
	\end{enumerate}
\end{proposition}

\begin{proof}

	For (ii), see e.g. \cite[Theorem II.1.2]{BH11}. For (i) and (iv), see
	\cite[Theorem II.1.4]{BH11} and \cite[Proposition 3.1]{ohta}. (iii) is
	folklore, for convenience we give a short proof.

	Clearly, $(M,\breve{d})$ is $\pi$-geodesic since $(M,d)$ is $(\pi /
		\sqrt{\kappa})$-geodesic. Choose any three distinct points $x,y,z \in M$
	such that the corresponding triangle has perimeter less than or equal $2
		\pi$ with respect to $\breve{d}$. Then that same triangle has perimeter
	less than or equal to $2 R_\kappa$ with respect to $d$, and, due to Lemma
	\ref{lem:reference-triangle}, there is a reference triangle for $(M,d)$
	formed by $x^\prime, y^\prime, z^\prime\in \mathbb{S}^2$ with respect to
	$d_\kappa$. In consequence, it forms also a reference triangle with respect
	to $d_1$ for $(M,\breve{d})$. With a geodesic $\breve\gamma$ from
	$y=\breve\gamma(0)$ to $z = \breve\gamma(\breve{d}(y,z))$ in
	$(M,\breve{d})$, we have that $t\mapsto \gamma(t) =
		\breve\gamma(t/\sqrt{\kappa})$ is a geodesic in $(M,d)$ from $y=
		\gamma(0)$ to $z= \gamma(d(y,z))$, and hence for every $0 \leq t \leq
		\breve d(y,z)$, we have
	\begin{align*}
		\breve{d}(x, \breve \gamma(t)) & = \sqrt{\kappa} \cdot d(x, \gamma(t \sqrt{\kappa}))   \leq  \sqrt{\kappa} \cdot d_\kappa(x^\prime, \gamma'(t\sqrt{\kappa})) = d_1(x^\prime, \breve\gamma'(t)).
	\end{align*}
	Here $\breve\gamma'$ is the geodesic from $y'= \breve\gamma'(0)$ to $z' = \breve\gamma'(\breve d(y,z))$ with respect to $d_1$, which reparametrizes to the geodesic  $t\mapsto \gamma'(t) = \breve\gamma'(t/\sqrt{\kappa})$ from $y'= \gamma'(0)$ to $z' = \gamma'(d(y,z))$ with respect to $d_\kappa$, for which the above inequality holds, by assumption. Thus, $(M, \breve{d})$ is indeed a CAT(1)-space.

\end{proof}

\begin{figure}[!h]
	\centering
	\subfloat[{\it Points involved in the \cat ~property.\label{fig:a}}]{\includegraphics[height=5cm,
			width=5cm]{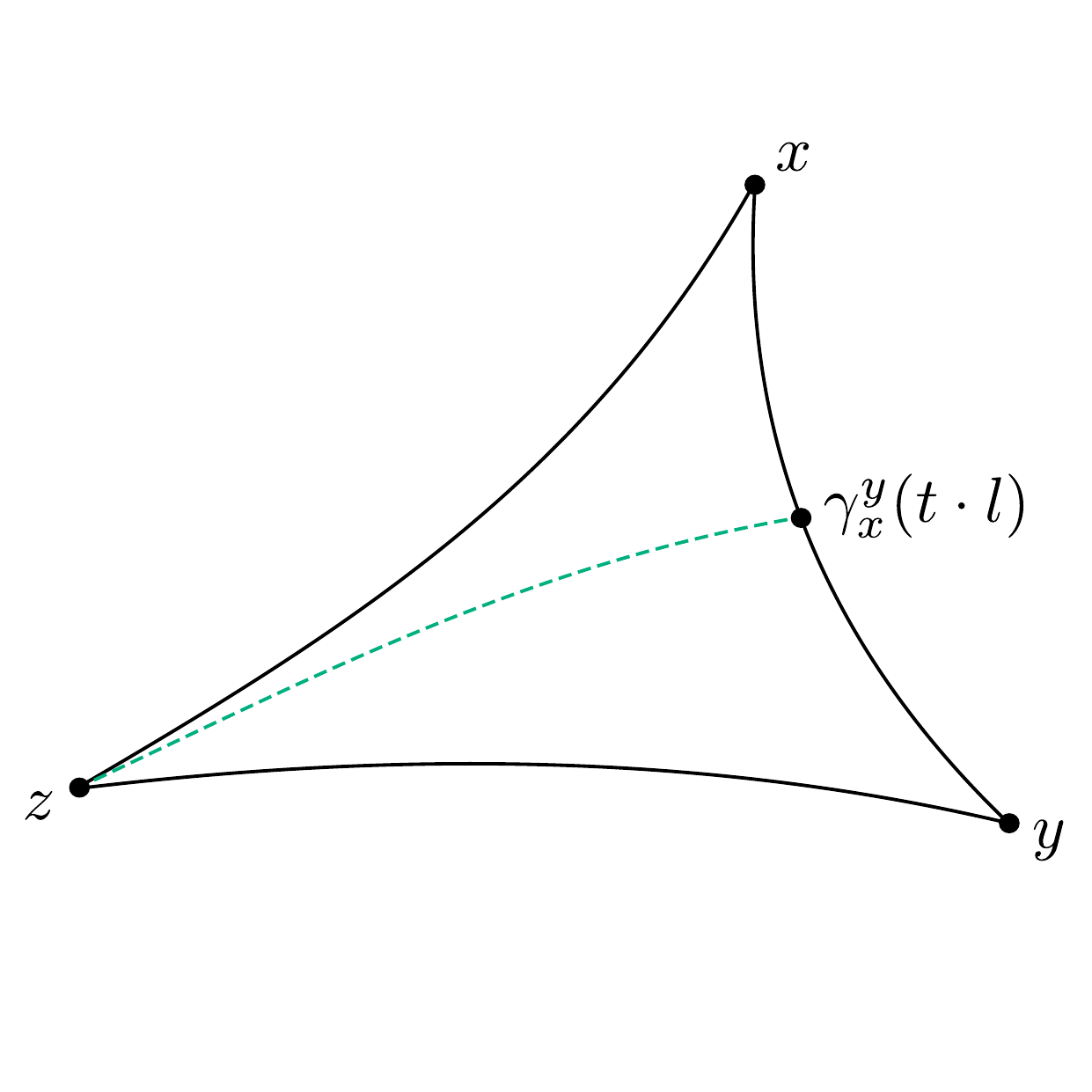}}\qquad
	\subfloat[Quadruple comparison.  \label{fig:b}]{\includegraphics[
			height=5cm, width=5cm]{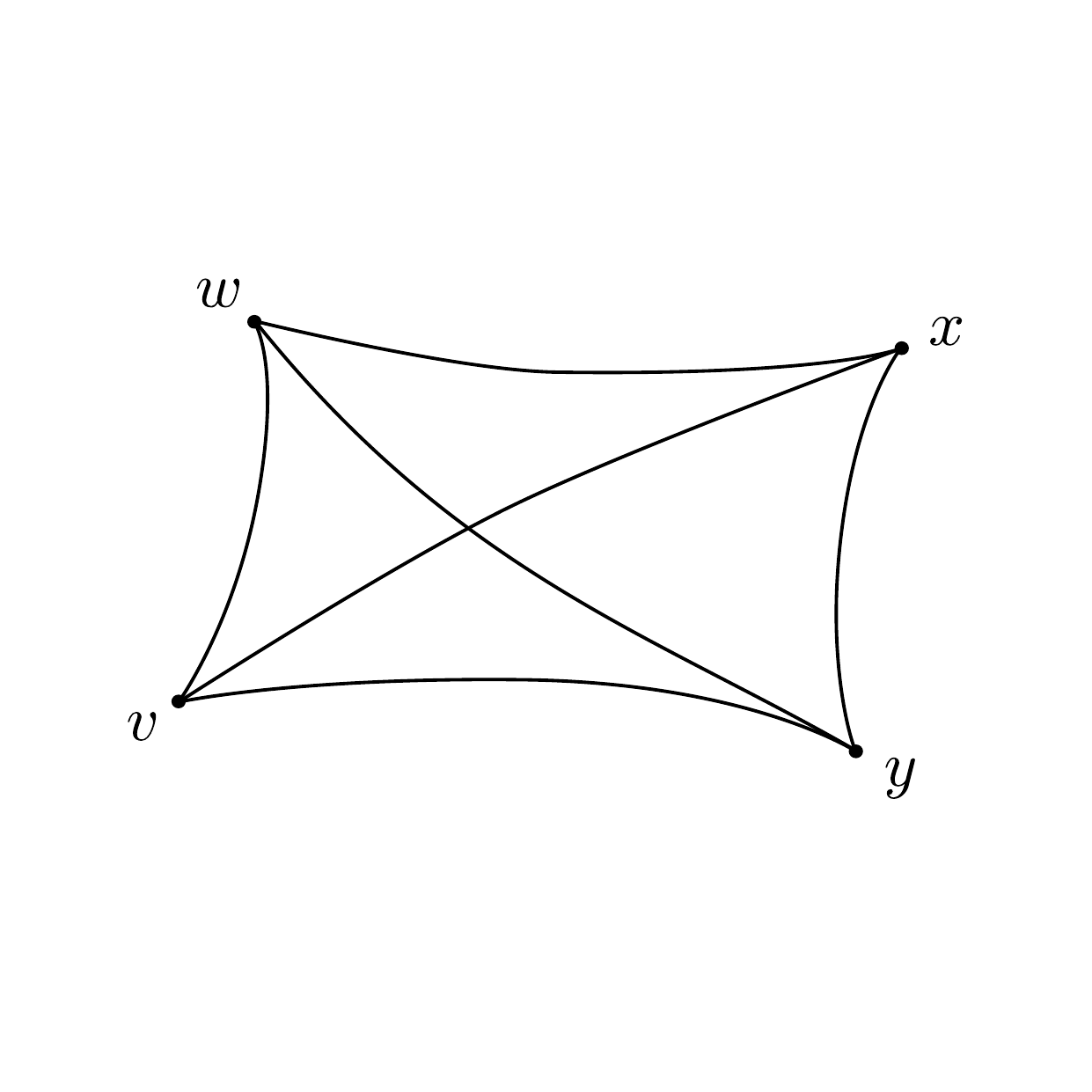}}\qquad

	\caption{\it Illustrations of Definition \ref{Def:CATk} and Proposition \ref{prop:quadcomp} \label{fig:npc-quad}}
\end{figure}

For $\kappa=0$, i.e. $C_{\kappa.\epsilon} = 1$, inequality (\ref{ineq:NPC}) is
also called the \emph{NPC inequality} in CAT(0)-spaces (standing for
nonpositive curvature). We also require the following inequality for Hadamard
spaces, see Figure \ref{fig:npc-quad}. 

\begin{proposition}[Reshetnyak's quadruple comparison]
	\label{prop:quadcomp}

	Let $(M,d)$ be a CAT(0) space. Then it holds for any $x,y,v,w \in M$ that
	\begin{align*}
		d^2(v,x) + d^2(w, y) \leq d^2(w, x) + d^2(v, y)
		+ 2 d(v, w) \cdot d(x, y).
	\end{align*}
\end{proposition}

\begin{proof}
	\cite[Proposition 2.4]{Sturm}.
\end{proof}

%
%
%
%
%
%

\subsection{Angles, Space of Directions and Tangent Cone}

In differentiable manifolds, one can construct at each point a vector space by
considering two curves equivalent if they pass through the point with the same
direction and speed.\\ In general \cat-spaces, the equivalence classes of
curves do not necessarily form a vector space but rather a cone. Formally, the
\emph{tangent cone} at each point is constructed by considering the cone over
the \emph{space of directions} of geodesics passing through the point $x$. Two
geodesics are then said to have the same direction if the \emph{angle} between
them is 0.

\begin{definition}
	\label{def:alexangle}

	Let $(M,d)$ be a \cat~space, $\kappa \in \RR$, and for $i=1,2$ let
	$\gamma_i : [0,a_i] \to M$, be two geodesics with $\gamma_i(0) = x \in M$
	and $a_i >0$. Then their \emph{Alexandrov angle} at $x$ is given by
	\begin{align*}
		\angle_x(\gamma_1, \gamma_2) = \arccos \lim_{t\searrow 0}
		\left( 1 - \frac{d^2(\gamma_1(t), \gamma_2(t))}{2 t^2}\right).
	\end{align*}
\end{definition}

Taking, as usual, the principal branch of the arccos in $[0,2\pi)$, it turns out that the Alexandrov angle is between $0$ and $\pi$. In a Euclidean
space, it is just the classical angle between two straight rays emanating from
$x$. It is also very well-behaved in \cat~spaces as the following teaches.

\begin{lemma}
	\label{lem:Alexandrov-angle-triangle}

	If $(M,d)$ is a \cat~space for some $\kappa \in\RR$, the Alexandrov angle
	is well-defined, it is symmetric, nonnegative and satisfies the triangle
	inequality, i.e. for $i=1,2,3$ and three geodesics $\gamma_i: [0,a_i] \to
		M$ with $ \gamma_i(0) =  x \in M$ and $a_i >0$,
	$$ \angle_x(\gamma_1,\gamma_3) \leq \angle_x(\gamma_1,\gamma_2)+\angle_x(\gamma_2,\gamma_3)\,.$$
\end{lemma}

\begin{proof}

	Firstly, we have $1 - d^2(\gamma_1(t), \gamma_2(t)) \leq 1$, and
	\begin{align*}
		1 - \frac{d^2(\gamma_1(t), \gamma_2(t))}{2 t^2} \leq 1\,,
	\end{align*}
	and
	\begin{align*}
		1 - \frac{d^2(\gamma_1(t), \gamma_2(t))}{2 t^2} & \geq 1 - \frac{\left(d(x, \gamma_2(t)) + d(x, \gamma_1(t))\right)^2}{2 t^2} \\
		                                                & =  1 - \frac{4 t^2}{2 t^2} = -1\,,                                          \\
	\end{align*}
	recalling that all of our geodesics have the same unit speed.
	Thus, the argument of $\arccos$ lies in $[-1,1]$ for all $0 \leq t \leq \min\{a_1, a_2\}$,
	and therefore the limit is in $[1,-1]$, if it exists; and,  indeed, it exists in \cat~spaces by \cite[Proposition 4.3.2]{burago}.
	Symmetry is obvious, and the triangle inequality has been shown in
	\cite[Theorem I.1.14]{BH11}.
\end{proof}

For a triangle in a \cat~space of suitably bounded perimeter, define in the
following the angles of a reference triangle in $(M_\kappa,d_\kappa)$.

\begin{definition}
	\label{def:compangle}

	Let $(M,d)$ be a \cat-space and $x,y,z \in M$ forming a triangle with
	perimeter less than $2 R_\kappa$. Then the \emph{comparison angle} of $y$
	and $z$ at $x$ is given by
	\begin{align*}
		\angle_x^\kappa(y,z)=
		\begin{cases}
			\arccos\left(\frac{\cosh(\sqrt{-\kappa}d(y,z) - \cosh(\sqrt{-\kappa}d(x,y)) \cosh({\sqrt{-\kappa}d(x,z)})}{\sinh(\sqrt{-\kappa}d(x,y)) \sinh({\sqrt{-\kappa}d(x,z)})}\right) \quad & \text{if } \kappa < 0, \\
			\arccos\left(\frac{d^2(x,y) + d^2(x,z) - d^2(y,z)}{2 d(x,y) d(x,z)}\right) \quad                                                                                                   & \text{if } \kappa = 0, \\
			\arccos\left(\frac{\cos(\sqrt{\kappa}d(y,z) - \cos(\sqrt{\kappa}d(x,y)) \cos({\sqrt{\kappa}d(x,z)})}{\sin(\sqrt{\kappa}d(x,y)) \sin({\sqrt{\kappa}d(x,z)})}\right) \quad           & \text{if } \kappa > 0. \\
		\end{cases}
	\end{align*}

\end{definition}

These angles form an upper bound, decreasing in the sense of the following
proposition.

\begin{proposition}
	\label{prop:anglelimit}

	Let $(M,d)$ be a \cat-space and for $i=1,2$ let $\gamma_i : [0,a_i] \to M$
	be two geodesics with $\gamma_i(0) = x \in M$ and $a_i >0$. Then, we have
	\begin{align*}
		\angle_{x}^\kappa(\gamma_1(t),\gamma_2(t)) \leq \angle_{x}^\kappa(\gamma_1(t^\prime),\gamma_2(t^\prime)) \mbox{ for }0 \leq t < t^\prime \leq \min\{a_1,a_2\}\,,
	\end{align*}
	as well as
	\begin{align*}
		\angle_{x}(\gamma_1,\gamma_2) = \lim_{t \searrow 0} \angle_{x}^\kappa(\gamma_1(t),\gamma_2(t)).
	\end{align*}
\end{proposition}

\begin{proof}

	\cite[Proposition II.3.1]{BH11}.
\end{proof}

\begin{definition}
	\label{def:spaceofdir}

	Let $(M,d)$ be a \cat~space for some $\kappa\in \RR$ and let $x \in M$.
	For $i=1,2$, two geodesics $\gamma_i : [0,a_i] \to M$ starting at
	$\gamma_i(0)=x$, $a_i>0$ have the same \emph{direction at $x$} if
	$\angle_x(\gamma_1, \gamma_2) = 0$. Then write $\beta \sim \gamma$. The
	\emph{space of directions at $x$} consists of equivalence classes
	$[\gamma] = \{\beta: \beta \sim\gamma\}$ of geodesics starting at $x$,
	i.e.
	\begin{align*}
		\Sigma_x M = \{[\gamma]:
		\text{$\gamma$ is a geodesic starting at $\gamma(0)=x$} \} \,.
	\end{align*}
	Further, the \emph{tangent cone} $T_x M$ at $x$ is the Euclidean cone over
	the space of directions at $x$, cf. Definition \ref{Def:Cones}, i.e.
	\begin{align*}
		T_x M = C_0 \left( \Sigma_x M \right).
	\end{align*}
	We simply write $\Sigma_x = \Sigma_x M$ and $T:x=T_x M$ if $M$ is self-understood.
\end{definition}

The following lemma shows that the Alexandrov angle is a metric for the space
of directions and thus the Euclidean tangent cone is well-defined. In slight
abuse of notation, for $\sigma,\tau \in \Sigma_x$, define $$\angle_x(\sigma,
	\tau) := \angle_x(\beta, \gamma)$$ where $\beta, \gamma$ are representatives
of $\sigma,\tau$, respectively. Indeed, the r.h.s. above is independent of the
representative chosen, for if $\delta \sim \gamma$, then $$\angle_x(\beta,
	\gamma) \leq \angle_x(\beta, \delta) \leq \angle_x(\beta, \gamma)$$ due to the
triangle inequality from Lemma \ref{lem:Alexandrov-angle-triangle}.

\begin{lemma} 
	\label{lemma:directions_cat}

	Let $(M,d)$ be a \cat~space for some $\kappa\in \RR$ and let $x \in M$. Then, $(\Sigma_x M, \angle_x)$
	is a well-defined metric space. Further, its completion is a CAT(1) space
	and the completion of the tangent cone is a CAT(0) space, i.e. a Hadamard space.
\end{lemma}

\begin{proof}

	The triangle inequality for the Alexandrov angle from Lemma
	\ref{lem:Alexandrov-angle-triangle} ensures that $\sim$ is indeed an
	equivalence relation and thus Lemma \ref{lem:Alexandrov-angle-triangle}
	yields the first assertion. For the second part, see \cite[Theorem II.3.19]{BH11}.
\end{proof}

The completion of $\Sigma_x$ will be denoted by $\cl(\Sigma_x)$ and $\cl(T_x)$ denotes the completion of $T_x$

%
%
%

Under additional assumptions, the space of directions and the tangent cone both inherit many
desirable properties from the original space.


\begin{definition}

	Let $(M,d)$ be an $r$-geodesic space, $r \in (0, \infty]$. A geodesic
	$\gamma: [a,b] \to M$ is called \emph{extendible} if there is another
	geodesic $\gamma^\prime : [a^\prime, b^\prime] \to M$ with $a^\prime < a$,
	$b < b^\prime$, and $\gamma = \gamma^\prime \lvert_{[a,b]}$.

	The space $(M,d)$ is called \emph{geodesically complete} if it is complete
	and any geodesic is extendible.
\end{definition}



\begin{proposition} 
	\label{prop:cat_dir}
	Let $(M,d)$ be a locally compact and geodesically complete \cat~space for some $\kappa\in \RR$ and let $x \in M$. Then,
	\begin{enumerate}
		\item[(i)] the space of directions $\Sigma_x$ is a compact, separable,
		      metrically and geodesically complete CAT(1) space, and
		\item[(ii)] the tangent cone $T_x$ is a locally compact, separable and
		      geodesically complete Hadamard space.
	\end{enumerate}
\end{proposition}
\begin{proof}
	\cite[Corollary 5.7 and 5.8]{LN18}.
\end{proof}

\section{Directional Stickiness and the Other Flavors in \cat~Spaces}\label{scn:dir-sticky}

For this entire section -- and the rest of this paper -- we assume that
\begin{center}
	$(M,d)$ is a \cat~space with $\kappa \in \RR$.
\end{center}
For specific results, later in this section we introduce the more restricted Scenario \ref{set:2} on page \pageref{set:2} and Scenario \ref{set:3} on page \pageref{set:3}.

Now, we first introduce additional notation for working on the space of directions and tangent cone  (for its completions, cf. Lemma \ref{lemma:directions_cat}), and a well known result on existence of Fr\'echet means.

\begin{definition}

	Let 
	$x,y \in M$, with $d(x,y) = l < R_\kappa$
	and $\sigma \in \Sigma_x$. Then
	\begin{enumerate}
		\item  $\gamma_x^y : [0, l] \to M$ denotes the unique (see Proposition \ref{prop:catgeod} (i)) geodesic (recall that all of our geodesics have unit speed) with
		      $\gamma_x^y(0) = x$ and $\gamma_x^y(l) = y$,
		\item define
		      \begin{align*}
			      \dir_x : \   B_{R_\kappa}(x)\setminus \{ x \} & \to \Sigma_x          \\
			      z                                             & \mapsto [\gamma_x^z],
		      \end{align*}
		      and 
		      the \emph{log-map at $x$} by
		      \begin{align*}
			      \log_x : \   B_{R_\kappa}(x) & \to        \cl(T_x) \\
			      z                            & \mapsto \left\{
			      \begin{array}{ll}
				      \big (\dir_x(z), d(x,z) \big ) & \quad \text{if } z \neq x, \\
				      \mathcal{O}                    & \quad \text{if } z = x,
			      \end{array}
			      \right.
		      \end{align*}
		\item for $\prb \in \cP(M)$,
		      $\widetilde{\prb} = \log_x \# \prb$ denotes the pushforward
		      measure under $\log_x$, which is a probability measure on $\cl(T_x)$ if $\prb(B_{R_\kappa}(x))=1$,
		\item for brevity's sake write $\angle_x(\sigma, y) = \angle_x(\sigma,
			      \dir_x(y))$.
	\end{enumerate}
\end{definition}

Below, we cite the well known result that uniqueness of the Fr\'echet mean is guaranteed for complete \cat ~spaces if either
$\kappa \leq 0$, or for $\kappa >0$, if the distribution is supported by an open
geodesic half ball.
For $q \geq 1$, $x \in M$ and $\kappa \in \RR$, let
\begin{align*}
	\Qkap[q] = \Big \{ \prb \in \wst[q]{M} \ : \ \supp(\prb) \subset \Bkap \Big \},
\end{align*}
where we use the convention $B_{\infty /2}(x) = M$.

\begin{proposition} 
	\label{prop:uniquefrech}
	Let $(M,d)$ be a \cat ~space for
	some $\kappa \in \RR$, $x \in M$ and $\prb\in \Qkap$. Then $\prb$ has a unique
	Fr\'echet mean $\mu \in \Bkap$, i.e. $\fb(\prb) = \{\mu\}$.

\end{proposition}

\begin{proof}
	For CAT(0), \cite[Proposition 4.3]{Sturm} in conjunction with Proposition \ref{prop:catgeod} and
	\cite[Theorem B]{yokota} for CAT($\kappa$) spaces with $\kappa >0$.
\end{proof}



%
%
%

\subsection{Directional Stickiness and the Pull}\label{scn:pull}

Directional stickiness will be defined via directional derivatives of the Fr\'echet function, and we will see in Theorem \ref{theorem:mainthm} that stickiness occurs if the directional derivatives are all \emph{positive}.

\begin{definition} 
	Let 
	$x \in M$ and $\sigma \in \Sigma_x$ with representative $\gamma \in \sigma$. For a function $f: M \to
		\RR$ define the directional derivative of $f$ at $x$ in direction $\sigma$ by
	\begin{align*}
		\nabla_\sigma f(x) = \frac{\diff}{\diff t} \Big \lvert_{t=0} f
		\circ \gamma(t)
	\end{align*}
	if it exists and is well-defined.
\end{definition}

Recall from Remark \ref{rmk:Frechet-diff} that for $\prb\in \wst{M}$ the Fr\'echet difference $x \mapsto F^y_P(x)$ is well-defined for any fixed $y\in M$ and that its derivatives not involving $y$ agree with the corresponding derivatives of $F_\prb$ for $\prb\in \wst[2]{M}$. Hence, for notational convenience, in slight abuse of notation, we write $\nabla_\sigma 	F_\prb$, also for $\prb \in \wst{M}$, when actually $\nabla_\sigma F^y_\prb$ for suitable $y\in M$ is meant.

\begin{theorem} 
	\label{theorem:pull}

	Let 
	$x\in M$ and $\prb \in \wst[1]{M}$ with $\prb(B_{R_\kappa}(x)) = 1$.
	Then $\nabla_\sigma F_\prb(x)$ is well-defined for all $\sigma \in \Sigma_x$, and of form $$ \nabla_\sigma F_\prb(x) = - \int_M d(x, z) \cdot \cos\left( \angle_x(\sigma, 			z)\right)\diff \prb(z) . $$
\end{theorem}

\begin{proof}
	For fixed $z\in M$ let $x$ be variable for the moment and let
	\begin{align*}
		d_z :  B_{R_\kappa}(z) \to [0,\infty),\quad x \mapsto d(x,z).
	\end{align*}
	Then, due to the \emph{first variation formula} from  \cite[Corollary 4.5.7
		and Remark 4.5.12]{BBI01}, $\nabla_{\sigma} d_{z}(x)$ exists for all $x\in B_{R_\kappa}(z)$ and is of form
	\begin{align*}
		\nabla_{\sigma} d_{z}(x)  = -\cos \left( \angle_x(\sigma, z) \right).
	\end{align*}
	In consequence, with
	$$f_t(z) := \frac{d^2(\gamma(t),z) -d^2(x,z) }{t} $$
	for some $\gamma \in \sigma$ with, $\gamma(0)=x$ and $\gamma(a) \in M$ for suitable $0<a< R_\kappa$, we have for arbitrary $y\in M$,
	$$ \lim_{t\downarrow 0} f_t(z) = \nabla_{\sigma} \left(d^2(x,z)-d^2(y,z)\right)  = - 2d(x,z) \cos \left( \angle_x(\sigma, z) \right)\,. $$
	Since
	\begin{align*}
		\left \lvert \frac{d^2(\gamma(t), z) - d^2(x,z)}{t} \right \rvert & \leq \frac{(d(\gamma(t), z) + d(x,z)) \cdot d(x, \gamma(t))}{t} \\ &\leq d(x, \gamma(t)) + 2 \cdot d(x,z) \leq a + 2 \cdot d(x,z)\,,
	\end{align*}
	cf. Remark \ref{rmk:Frechet-diff}, where the right-hand side is integrable due to the assumption that $\prb \in \wst[1]{M}$, by dominated convergence, integration and differentiation below can be exchanged (e.g.
	\cite[p. 232]{fleming}), yielding indeed
	\begin{align*}
		\nabla_{\sigma} F_\prb(x)  = \nabla_{\sigma} F_\prb^y(x) & = \frac{1}{2}\, \nabla_{\sigma} \int_M \left(d^2(x,z)-d^2(y,z)\right)\diff\prb(z) \\
		                                                         & = - \int_M d(x, z) \cdot \cos\left( \angle_x(\sigma, z)\right)\diff \prb(z) .
	\end{align*}

\end{proof}

Next, we note Lipschitz continuity for the directional derivative of the Fr\'echet function.
\begin{corollary} 
	\label{corollary:contfrech}
	Let 
	$x\in M$ and $\prb \in \wst[1]{M}$ with $\prb(B_{R_\kappa}(x)) = 1$.
	Then we have for any $\sigma, \tau \in \Sigma_x$
	\begin{align*}
		\lvert \nabla_\sigma F_\prb(x) - \nabla_\tau F_\prb(x) \rvert \leq L_x \angle_x(\sigma, \tau),
	\end{align*}
	where
	\begin{align*}
		L_x = \int_M d(x, z) \ \diff \prb(z).
	\end{align*}
\end{corollary}

\begin{proof}
	We have by Theorem \ref{theorem:pull} that
	\begin{align*}
		\lvert \nabla_\sigma F_\prb(x) - \nabla_\tau F_\prb(x) \rvert & \leq
		\int_M d(x,z)\,\big \lvert \cos \angle_x(\sigma,z) -  \cos \angle_x(\tau,z)\big\vert\,\diff\prb(z)
		\\
		                                                              & \leq
		\int_M d(x,z)\,\big \lvert \angle_x(\sigma,z)  - \angle_x(\tau,z)\big\vert\,\diff\prb(z) \\
		                                                              & \leq
		\int_M d(x,z)\, \angle_x(\sigma,\tau)\,\diff\prb(z)
	\end{align*}
	where we used that $\cos$ is Lipschitz-1 and the triangle inequality from Lemma \ref{lem:Alexandrov-angle-triangle}. This yields the assertion.
\end{proof}

Before defining the fourth sticky flavor further below, we explore the relationship between nonnegative directional derivatives and the Fr\'echet mean.
The following result will be used frequently, in particular in as part of Theorem \ref{theorem:meanderivs}.

The case for \cat-spaces with $\kappa \geq 0$ is discussed in
\cite[Corollary 15]{yokota}. For convenience, we include the simple proof.

\begin{theorem} 
	\label{theorem:meanderivsI}

	If $\mean \in \mathfrak{b}(\prb)$ for some $\prb\in \wst{M}$ with $\prb(B_{R_\kappa}(\mu)) = 1$, then
	\begin{align*}
		\nabla_\sigma F_\prb(\mean) \geq 0\mbox{ for all }\sigma\in \Sigma_\mu\,.
	\end{align*}
\end{theorem}

\begin{proof}

	Suppose $\mean \in \mathfrak{b}(\prb)$ and $\sigma \in \Sigma_\mean$. If
	$\gamma \in \sigma$ then $F_\prb(\gamma(t)) \geq F_\prb(\mean)$ for all
	$t\geq 0$ for which $\gamma$ is defined, and therefore
	\begin{align*}
		\nabla_\sigma F_\prb(\mean) = \lim_{t \searrow 0}\frac{F_\prb(\gamma(t)) - F_\prb(\mean)}{t} \geq 0.
	\end{align*}
\end{proof}

We now give a converse of Theorem \ref{theorem:meanderivsI} that will also be frequently used later, notably in the proof of Theorem \ref{theorem:mainthm}. The case for \cat-spaces with $\kappa \geq 0$ is discussed in \cite[Corollary 	15]{yokota}. For convenience, we include the simple proof. To this end,
we require the following scenario additional to Scenario \ref{set:1}.

\begin{setting}
	\label{set:2}

	In addition to Scenario \ref{set:1}, assume that $(M,d)$ is a
	\cat ~space for some $\kappa \in \RR$. Furthermore, we restrict ourselves to $\prb \in  \Qkapn$ and to the case
	$S = \{ \mu \} \subset \Bkapn$ for some $x_0 \in M$.
\end{setting}

\begin{theorem} 
	\label{theorem:meanderivs}

	In Scenario \ref{set:2} we have
	$\{\mean\} =  \mathfrak{b}(\prb)$ if and only if
	\begin{align*}
		\nabla_\sigma F_\prb(\mean) \geq 0 \quad \forall \sigma \in \Sigma_\mean.
	\end{align*}
\end{theorem}

\begin{proof}

	The implication $\{\mean\} = \mathfrak{b}(\prb) \Rightarrow \nabla_\sigma
		F_\prb(\mean) \geq 0$ has already been shown in Theorem
	\ref{theorem:meanderivsI}.

	To see the converse implication, suppose that $\nabla_\sigma F_\prb(\mean)
		\geq 0$ for all $\sigma \in \Sigma_\mean$. Taking any point $y \in M
		\setminus \{\mean\}$ sufficiently nearby $\mean$ with direction $\sigma
		\in \Sigma_\mean$, we know for $\gamma \in \sigma$ by Proposition
	\ref{prop:catgeod} that $t \mapsto d^2(\mean, \gamma(t))$ and,
	consequently, $t \mapsto F_\prb(\gamma(t))$ is convex for $t\in
		[0,d(\mean,y)]$. Hence,
	\begin{align*}
		F_\prb(y) \geq F_\prb(\mean) + d(\mean,y) \cdot \nabla_{\sigma} F_\prb(\mean),
	\end{align*}
	see e.g. \cite[Theorem 23.2]{rockafellar}. Therefore, $F_\prb(\mean)
		\leq F_\prb(y)$ so that $\mu$ is a local minimizer for $F_\prb$. Due to
	Proposition \ref{prop:uniquefrech}, $\mean$ is the unique Fr\'echet mean.
\end{proof}

\begin{remark}

	Even if $\prb \in \Qkap$, there might still be other local minima of $F_\prb$ outside of $\Bkap$ if $\kappa >0$.
	For an example, see Figure \ref{fig:halfball-but-2ndlocalmin}. 
\end{remark}

\begin{figure}
	\centering
	\subfloat{\includegraphics[height=5cm,
			width=5cm]{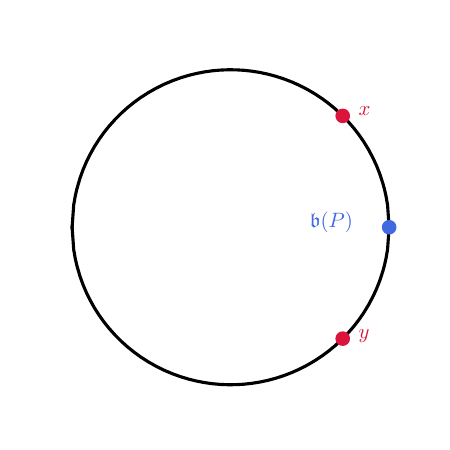}}\qquad
	\subfloat{\includegraphics[
			height=5cm, width=5cm]{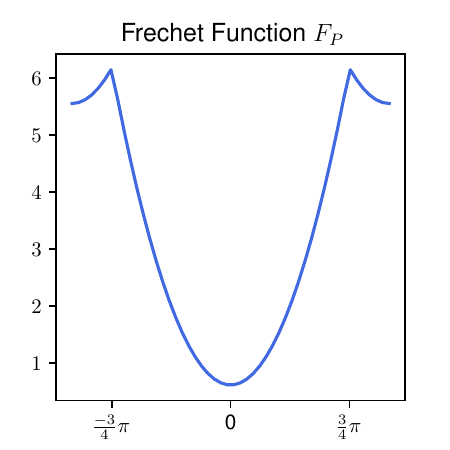}}\qquad

	\caption{\it The Fr\'echet function of the distribution $\prb = \frac{1}{2}(\delta_x + \delta_y)$ on the
		unit circle has a local minimum at the antipode of the Fr\'echet mean.}
	\label{fig:halfball-but-2ndlocalmin}
\end{figure}

\begin{definition}[Directional stickiness]

	In Scenario \ref{set:2}, a distribution $\prb \in \Qkapn$ with unique Fr\'echet mean $\mean$ is \emph{directionally sticky} if $\nabla_\sigma
		F_\prb(\mean) > 0$ for all directions $\sigma \in \Sigma_\mean$.
\end{definition}

It will be decisive to study the integrand of the derivative of the Fr\'echet
function.

\begin{definition}[The Pull] 
	\label{def:pull}
	Let 
	$x \in M$, $z \in
		B_{R_\kappa}(x)$ and $\sigma \in \Sigma_x$. Then the \emph{pull
		$\phi_{x,\sigma}$ of $z$ in direction $\sigma$} is given by
	\begin{equation}
		\label{eq:phi}
		\begin{split}
			\phi_{x, \sigma}(z) & = - \frac{1}{2} \nabla_\sigma d_z^2(x)                 \\
			                    & = d(x, z) \cdot \cos\left( \angle_x(\sigma, z)\right).
		\end{split}
	\end{equation}
\end{definition}

\begin{remark}\label{remark:derivative-Frechet-fcn-with-pull}
	For $\prb \in \wst[1]{M}$ with $\prb(\Bkapf)=1$, due to Theorem \ref{theorem:pull}, we can then write
	\begin{align*}
		\nabla_\sigma F_\prb(x) & =  - \int_M \phi_{x,\sigma}(z) \ \text{d}\prb(z). \\
	\end{align*}
\end{remark}

\begin{remark}\label{remark:folded-moments-directional-derivatives}
	Note that the directional derivatives of the Fr\'echet function are equal to the negative  \emph{first folded moments} for open books and
	the kale that have been detailed in Examples \ref{ex:open-books} and \ref{ex:kale}.
\end{remark}

\begin{remark}
	\label{rmk:pull-tangent-pull}

	Recall from Definition \ref{def:spaceofdir} that the tangent cone $T_x$ at
	$x\in M$ is the Euclidean cone over $\Sigma_x$, and hence $\Sigma_x$ can be
	naturally isometrically embedded into $\Sigma_\mathcal{O} \cl(T_x)$. 
	Thus, in the setting of Definition \ref{def:pull}, we have the \emph{tangent-pull}
	$\widetilde{\phi}_{\mathcal{O},\sigma}$   of $\log_x z$ at $\cO$ in
	direction of $(\sigma,1)$ given by
	\begin{align*}
		\widetilde{\phi}_{\mathcal{O}, \sigma}(\log_x z) & = -\frac{1}{2} \frac{\diff}{\diff t} \widetilde{d}^2(\log_x(z), \gamma_\mathcal{O}^{(\sigma,1)}(t)) \big \lvert_{t = 0},
	\end{align*}
	where $\widetilde{d}$ denotes the distance on the Euclidean cone $T_x$ and
	$\gamma_\mathcal{O}^{(\sigma,1)}$ is the straight line $t\mapsto (t,\sigma)$.
	With the notation of Definition \ref{Def:Cones}, we have $d_0=\widetilde{d}$
	and $d_{\Sigma_x}(\sigma,\tau) = d_\pi(\sigma,\tau) = \angle_x(\gamma_x^\sigma,
		\gamma_x^\tau))$
	In consequence, while in general $\widetilde{d}(\log_x(z),
		\gamma_\mathcal{O}^{(\sigma,1)}(t)) \neq d(z, \gamma_x^\sigma(t))$, both
	pull and corresponding tangent pull are equal: $$ \phi_{x,\sigma}(z) =
		d(x, z) \cdot \cos\left( \angle_x(\sigma, z)\right) =
		\widetilde{\phi}_{\mathcal{O}, \sigma}(\log_x z)\,.$$
\end{remark}

Thus, directional stickiness extends naturally to the tangent cone and from there
back to the original space.

\begin{theorem} 
	\label{thm:directional-sticky-tangent-cone}
	In Scenario \ref{set:2} consider $\prb\in \Qkapn$ with unique Fr\'echet mean $\mu \in \Bkapn$. Then $\cO$ is the unique Fr\'echet mean of
	$\log_\mean \# \prb \in \wst{\cl(T_\mean)}$.
	If $\Sigma_\mean = \cl(\Sigma_\mean)$,
	$\log_x \# \prb$ is directionally sticky if and only if $\prb$ is so.
\end{theorem}

\begin{proof}

	In consequence of the equality of pull and tangent pull by the above Remarks \ref{remark:derivative-Frechet-fcn-with-pull} and \ref{rmk:pull-tangent-pull}, we have for all $\sigma \in \Sigma_\mean$
	\begin{align}
		\label{eq:directional-derivs-tangent-cone}
		\nabla_\sigma F_\prb(\mean) = \nabla_\sigma F_{\log_\mean \# \prb}(\mathcal{O}).
	\end{align}
	By hypothesis, due to Theorem \ref{theorem:meanderivs}, the l.h.s. is
	$\geq 0$, hence, so is the r.h.s.
	Moreover, for any sequence $(\tau)_{{i \in \NN}} \in \Sigma_\mean$ converging against $\tau \in \cl(\Sigma_\mean)$
	we have by Remark \ref{remark:folded-moments-directional-derivatives} and Corollary \ref{corollary:contfrech} that
	$\nabla_\tau F_{\log_\mean \# \prb}(\mathcal{O}) \geq 0$.
	Further, by Lemma \ref{lemma:directions_cat}, $\cl(T_{\mean})$ 
	is a Hadamard space.
	Hence, $\cO$ is the Fr\'echet mean of $\log_x \# \prb$ by Theorem \ref{theorem:meanderivs}.
	In addition, (\ref{eq:directional-derivs-tangent-cone}) shows that $\prb$ is
	directionally sticky at $\mean$ if and only if $\log_\mean \# \prb$ is
	directionally sticky at $\mathcal{O}$.
\end{proof}

\begin{remark}
	\label{rmk:pull-change of curvature}
	Setting  $\breve{d}=\sqrt{\kappa}d$ in case of $\kappa > 0$, we have seen in Proposition \ref{prop:catgeod} that $(M, \breve{d})$ is a CAT(1) space. In particular, geodesics agree in both metrics, $d$ and $\breve{d}$, with one another after suitable rescaling (recall that our geodesics are of unit speed). In consequence their Alexandrov angles agree in both metrics, and so, the spaces of directions can be identified with one another. Letting $x,z\in M$, $\sigma \in \Sigma_x$, note that $d(x,z) < R_\kappa$ if and only if $\breve{d}(x,z) < R_1$.
\end{remark}

\begin{lemma}
	\label{lem:pull-change of curvature}
	Assume that 
	$\kappa > 0$ and let
	$x \in M$, $z \in B_{R_\kappa}(x)$, $\sigma \in \Sigma_x$ and $\breve{d}=\sqrt{\kappa}d$.
	Then $(M,\breve{d})$ is a CAT(1) space and for the corresponding pulls we have
	$$\phi_{x,\sigma}(z) = \frac{1}{\sqrt{\kappa}}~ \breve{\phi}_{x, \sigma}(z)$$
	where $\breve{d}(x,z) < R_1$.
\end{lemma}

\begin{proof}
	With Remark \ref{rmk:pull-change of curvature}, if  $\gamma \in\sigma$ with respect to $d$, then $\breve{\gamma} \in \sigma$ with respect to $\breve{d}$ where $\breve{\gamma}(t) = \gamma(t / \sqrt{\kappa})$. Hence, we have
	\begin{align*}
		\phi_{x,\sigma}(z) & = -\lim_{t \searrow 0} \frac{d^2(\gamma(t),z)- d^2(x,z)}{t}                                                                             \\
		                   & = -\frac{1}{\kappa}\lim_{t \searrow 0} \frac{\breve{d}^2(\gamma(t),z)- \breve{d}^2(x,z)}{t}                                             \\
		                   & = -\frac{1}{\sqrt{\kappa}} \lim_{t \searrow 0} \frac{\breve{d}^2(\breve{\gamma}(\sqrt{\kappa} t),z)- \breve{d}^2(x,z)}{\sqrt{\kappa} t}
		\\
		                   & = \frac{1}{\sqrt{\kappa}} \breve{\phi}_{x, \sigma}(z).
	\end{align*}
	In conjunction with Remark \ref{rmk:pull-change of curvature}, this yields the assertions.
\end{proof}

\begin{lemma} 
	\label{lemma:philip}
	Let 
	$x\in M$ and $\sigma\in \Sigma_x$. Further, in case of $\kappa>0$, let $0<\epsilon < \pi$ be arbitrary and $R= R_\kappa - \epsilon/\sqrt{k}$ and in case of $\kappa \leq 0$ let $R=\infty$. Then there is a constant $C_\kappa(\epsilon) >0$ such that
	\begin{align*}
		\lvert \phi_{x,\sigma}(z) - \phi_{x,\sigma}(y) \rvert \leq C_\kappa(\epsilon)\, d(y, z)\mbox{ for all }y,z \in B_{R}(x).
	\end{align*}
	In particular $C_\kappa(\epsilon) = 1$ for $\kappa \leq 0$ and
	$$C_\kappa(\epsilon) = \max_{a,b,\theta \in [0,\pi-\epsilon]^2\times [0,\pi]} \Psi(a,b,\theta)$$
	for $\kappa > 0$, with the continuous function
	\begin{align*}
		\Psi : & \ [0, \pi-\epsilon]^2 \times [0, \pi] \to (0,\infty), \\
		       & (a,b,\theta) \mapsto
		\begin{cases}
			1 \quad                                                                                                       & \textrm{if } a = 0 = b, \text{ or } \theta = 0 \text{ and }a=b, \\
			\frac{\sqrt{a^2 + b^2 - 2ab\cos(\theta)}}{\arccos\left(\sin(a) \sin(b) \cos(\theta) + \cos(a) \cos(b)\right)} & \textrm{else.}
		\end{cases}
	\end{align*}

\end{lemma}

\begin{proof}
	First we do the case $\kappa \leq 0$. According to Proposition \ref{prop:catgeod} we can assume $\kappa =0$.

	With arbitrary $\gamma \in\sigma$
	apply the quadruple comparison
	(Proposition \ref{prop:quadcomp}), yielding for $t\geq 0$ sufficiently
	small,
	\begin{align*}
		d^2(x, y) + d^2(\gamma(t), z) - d^2(\gamma(t), y) - d^2(x, z)
		\leq 2 \cdot d(y,z) \cdot d(x, \gamma(t)),
	\end{align*}
	and
	\begin{align*}
		d^2(\gamma(t), y) + d^2(x, z) - d^2(x, y) - d^2(\gamma(t), z)
		\leq 2 \cdot d(y,z) \cdot d(x, \gamma(t)).
	\end{align*}
	Hence, it holds that
	\begin{align*}
		\Big\lvert \left(d^2(\gamma(t), z) - d^2(x, z)\right) -
		\left( d^2(\gamma(t), y) - d^2(x, y) \right) \Big \rvert
		\leq 2 \cdot d(x, \gamma(t)) \cdot d(y,z).
	\end{align*}
	Thus, we have that
	\begin{align*}
		\lvert \phi_{x,\sigma}(z) - \phi_{x,\sigma}(y) \rvert
		 & = \frac{1}{2} \lim_{t \searrow 0} \frac{\Big\lvert
			\left(d^2(\gamma(t), z) - d^2(x, z)\right) - \left(
		d^2(\gamma(t), y) - d^2(x, y) \right)  \Big \rvert }{t}            \\
		 & \leq \lim_{t \searrow 0} \frac{d(x, \gamma(t)) \cdot d(y,z)}{t} \\
		 & =  \lim_{t \searrow 0} \frac{t \cdot d(y,z)}{t}                 \\
		 & = d(y,z),
	\end{align*}
	as asserted.

	Next, we do the case $\kappa =1$.

	Taking into account that the argument of the $\arccos$ in the denominator
	of $\Psi$ is the inner product of the unit length vectors $(\cos(a),
		\sin(a) \cos(\theta), \sin(a) \sin(\theta))$ and $(\cos(b),\sin(b),0)$,
	which is less than one except for $a = 0 = b$, $a=-\pi =b$ or $\theta = 0$
	and $a=b$, series expansions there yield that the function $\Psi$ is
	continuous on $[0, \pi-\epsilon]^2\times [0, \pi]$ for every $0<\epsilon <
		\pi$, and hence, it must attain a maximum $C:=C_1(\epsilon)$. Further, due to
	Lemma \ref{lemma:directions_cat} the
	tangent cone $T_x$
	is a Hadamard
	space. Recalling from Remark \ref{rmk:pull-tangent-pull} that $\widetilde{d}$
	denotes the distance on the tangent cone, application of Remark
	\ref{rmk:pull-tangent-pull} and the case, already shown for $\kappa =0$, yields,
	taking into account Definition \ref{Def:Cones},
	\begin{align}
		\nonumber
		\lvert \phi_{x,\sigma}(z) - \phi_{x,\sigma}(y) \rvert & = \lvert \widetilde{\phi}_{\mathcal{O},\sigma}(\log_x(z)) - \widetilde{\phi}_{\mathcal{O},\sigma}(\log_x(y)) \rvert \\ \nonumber
		                                                      & \leq \widetilde{d}(\log_x(y), \log_x(z))                                                                            \\ \nonumber
		                                                      & =  \sqrt{d_x^2(y) + d_x^2(z) - 2 d_x(y) d_x(z) \cos(\angle_x(y,z))}                                                 \\ \nonumber
		                                                      & \leq  C\cdot \arccos\Big(\sin(d_x(y)) \sin(d_x(z)) \cos(\angle_x(y,z))
		\\ \label{eq:pull-difference-dist1}
		                                                      & \hspace*{3cm }+ \cos(d_x(y)) \cos(d_x(z)) \Big).
	\end{align}
	Further, due to Proposition \ref{prop:anglelimit} and Definition \ref{def:compangle}, obtain
	\begin{align*}
		\angle_x(y,z) \leq \angle_x^1(y,z) =  \arccos \left(\frac{\cos(d(y,z)) - \cos(d_x(y))\cos(d_x(z))}{\sin(d_x(y))\sin(d_x(z))} \right).
	\end{align*}
	Setting $\theta:= \angle_x^1(y,z)$, since $\cos$ is non-increasing on $[0,\pi]$, we have $\cos(\angle_x(y,z)) \geq \cos(\theta)$.
	As $\arccos$ is also non-increasing, we obtain
	\begin{align*}
		 & \arccos \left( \sin(d_x(y)) \sin(d_x(z)) \cos(\angle_x(y,z)) + \cos(d_x(y)) \cos(d_x(z)) \right) \\
		 & \leq \arccos \left( \sin(d_x(y)) \sin(d_x(z)) \cos(\theta) + \cos(d_x(y)) \cos(d_x(z)) \right)   \\
		 & = d(y,z).
	\end{align*}
	Combining this with (\ref{eq:pull-difference-dist1}) we have indeed
	\begin{align*}
		\lvert \phi_{x,\sigma}(z) - \phi_{x,\sigma}(y) \rvert & \leq C\cdot \arccos\Big(\sin(d_x(y)) \sin(d_x(z)) \cos(\angle_x(x,y))
		\\
		                                                      & \hspace*{3cm }+ \cos(d_x(y)) \cos(d_x(z)) \Big)
		\\
		                                                      & \leq C \cdot d(y,z)\,,
	\end{align*}
	as asserted.

	Finally, we reduce the general case $\kappa > 0$ to the case $\kappa =1$.
	Indeed, using the notation from Lemma \ref{lem:pull-change of curvature}
	\begin{align*}
		\vert \phi_{x,\sigma}(z) -  \phi_{x,\sigma}(z)\vert & = \frac{1}{\sqrt{\kappa}}~ \vert \breve{\phi}_{x,\sigma}(z) -  \breve{\phi}_{x,\sigma}(z)\vert
		\\
		                                                    & \leq \frac{1}{\sqrt{\kappa}}~ C\,\breve{d}(y,z) = C d(y,z)\,.
	\end{align*}
\end{proof}

\begin{theorem} 
	\label{theorem:derivcontr_sphere}

	In Scenario \ref{set:2} we have
	for $\prb, \prbalt \in \Qkapn$,
	\begin{align*}
		\lvert \nabla_\sigma F_ \prb (\mean)- \nabla_\sigma F_\prbalt(\mean) \rvert
		\leq C_\kappa(\epsilon) \cdot W_1(\prb, \prbalt),
	\end{align*}
	with the constant $C_\kappa(\epsilon)$ from Lemma \ref{lemma:philip}.
\end{theorem}

\begin{proof}
	Let $\varpi$ denote an optimal coupling of $\prb$ and $\prbalt$ for the cost
	function $d$. Since $\prb(B_{R_\kappa}(\mean)) = 1 = \prbalt(B_{R_\kappa}(\mean))$ in Scenario \ref{set:2}, by Theorem \ref{theorem:pull} and Lemma \ref{lemma:philip}, we have
	\begin{align*}
		\Bigg\lvert \nabla_\sigma F_{\prb} (\mean) - \nabla_\sigma F_{\prbalt} (\mean)
		\Bigg \rvert & = \Bigg \lvert \int_M \phi_{\mean,\sigma}(z) \ \text{d}\prb(z)
		- \int_M \phi_{\mean,\sigma}(y) \ \text{d}\prbalt(y)\Bigg \rvert                                     \\
		             & = \Bigg \lvert \int_{M \times M} \left(\phi_{\mean,\sigma}(z) -
		\phi_{\mean,\sigma}(y) \right) \ \text{d}\varpi(y, z)\Bigg \rvert                                    \\
		             & \leq  \int_{M \times M} \lvert \phi_{\mean,\sigma}(z) - \phi_{\mean,\sigma}(y) \rvert
		\ \text{d}\varpi(y, z)                                                                               \\
		             & \leq \int_{M \times M}  C_\kappa d(y,z) \ \text{d}\varpi(y,z)                         \\
		             & = C_\kappa W_1(\prb,\prbalt).
	\end{align*}

\end{proof}

Perturbation and sample stickiness along with the following regularity condition imply
directional stickiness.

\begin{definition}
	We say that $\prb \in \Qkap$ with $\fb(\prb) = \mean$ fulfills the \emph{pull condition} (pc) if there is \emph{no} direction $\sigma \in \Sigma_\mean$ such that the pull $\phi_{\mean,\sigma}$ (see Equation \ref{eq:phi}) vanishes $\prb$-almost everywhere, i.e.
	\begin{equation}
		\label{pullcond}
		\forall \sigma \in \Sigma_\mean: \ \prb\Big( \phi_{\mean, \sigma}^{-1}(\{0\})\Big) < 1 . \tag{pc}
	\end{equation}
\end{definition}

\begin{remark}
	Assumption (pc)
	ensures that the measure lives on a large
	enough subspace. In particular, it rules out that the measure is
	just a point mass at $\mean$. The distribution considered in Example \ref{ex:samplenotW1} provides a non-trivial counterexample
	where the assumption is violated.
\end{remark}

\begin{theorem} 
	\label{theorem:mainthmdir}
	In Scenario \ref{set:2} for $\prb\in \Qkapn$  each of the following statements
	\begin{enumerate}
		\item[(P)] $\prb$ is $\Bkapn$-perturbation sticky at $\mean$,
		\item[(Spc)] $\prb$ is sample sticky at $\mean$ and fulfills the pull condition (\ref{pullcond}),
	\end{enumerate}
	implies
	\begin{enumerate}
		\item[($\nabla$)] $\prb$ is directionally sticky at $\mean$.
	\end{enumerate}
\end{theorem}

\begin{proof}

	``(P) $\Rightarrow$ ($\nabla$)'': We have $\nabla_\sigma F_{\prb} (\mean) \geq 0$ for all $\sigma \in \Sigma_{\mean}$ by Theorem \ref{theorem:meanderivs}, so we need to show that $\nabla_\sigma F_{\prb} (\mean) \neq 0$ for arbitrary $\sigma \in \Sigma_{\mean}$. Choose $\gamma_\mu^y \in \sigma$ for suitable $y\in M$. By assumption $P$ is perturbation sticky, hence there is $0 < t_y < 1$ such that $\mean$ is the unique mean of
	\begin{align*}
		\prbalt_t = (1-t) \cdot \prb + t \cdot \delta_y, \quad t \in [0,1]\,.
	\end{align*}
	In consequence, for all $0 < t \leq t_y$, again by Theorem \ref{theorem:meanderivs},
	\begin{align*}
		0 \leq & \nabla_\sigma F_{\prbalt_t} (\mean)                                                    \\
		=      & (1-t) \cdot \nabla_\sigma F_{\prb} (\mean)+ t \cdot \nabla_\sigma F_{\delta_y} (\mean) \\
		=      & (1-t) \cdot \nabla_\sigma F_{\prb} (\mean) + \frac{t}{2} \cdot \frac{d}{ds}
		d^2(y,\gamma_{\mean}^y(s)) \Big\lvert_{s=0}                                                     \\
		=      & (1-t) \cdot \nabla_\sigma F_{\prb} (\mean) + \frac{t}{2} \cdot \frac{d}{ds} \left(
		(d(\mean,y)-s)^2 \right) \Big\lvert_{s=0}                                                       \\
		=      & (1-t) \cdot \nabla_\sigma F_{\prb} (\mean) -t \cdot d(\mean,y),
	\end{align*}
	where the penultimate equality holds as $\gamma_{\mean}^y$ is a geodesic from $\mean$ to $y$, which has unit speed. Hence,
	\begin{align*}
		\nabla_\sigma F_{\prb} (\mean) \geq \frac{t}{1-t} \cdot d(\mean,y) > 0,
	\end{align*}
	yielding ``($\nabla$)''.

	``(Spc) $\Rightarrow$ ($\nabla$)'': By assumption, there is a.s. random $N$ such that
	\begin{equation*}
		\mathfrak{b}(\prb_n) = \{\mean\} \quad \forall n \geq N.
	\end{equation*}
	By the strong law of large numbers (see \cite[Theorem 6.3, Propostion 6.6]{Sturm} and \cite[Proposition 24]{yokota}),
	we have $\mu \in \mathfrak{b}(\prb)$ and since $\prb\in \cQ$ has a unique mean by assumption,
	$$\{\mu\} = \mathfrak{b}(\prb)\,.$$
	Now, suppose that ($\nabla$) was false. Then there is $\sigma \in \Sigma_{\mean}$ with $\nabla_\sigma F_\prb(\mean) \leq 0$, and hence
	\begin{equation}
		\label{eq:zeroderiv}
		\nabla_\sigma F_\prb(\mean) = 0\,,
	\end{equation}
	by Theorem \ref{theorem:meanderivs}. Due to sample stickiness, again by Theorem \ref{theorem:meanderivs}, we have
	\begin{align}
		\label{ineq:empdeg}
		Y_n = n \cdot \nabla_\sigma F_{\prb_n}(\mean) = -\sum_{i = 1}^n \phi_{\mean,\sigma}(X_i) \geq 0 \quad
		\forall n \geq N,
	\end{align}
	where, due to (\ref{eq:zeroderiv}) and definition of the pull, we have
	\begin{align*}
		\mathbb{E}(\phi_{\mean,\sigma}(X_i)) = -\nabla_\sigma F_\prb(\mean) = 0 \quad \forall i \in \NN.
	\end{align*}


	It follows that the random walk $(Y_n)_{n \in \mathbb{N}}$ is reversible, see
	\cite[Theorem 5.4.8]{durrett}. Thus, with probability 1, we have that either $Y_n = 0$ for all $n \in
		\mathbb{N}$, or $- \infty = \liminf_n Y_n < \limsup_n Y_n = \infty$,
	see e.g. \cite[Exercise 5.4.1]{durrett}. The first case cannot occur
	since, due to (Spc), there is no direction $\sigma \in \Sigma_\mean$ such
	that $\phi_{\mean,\sigma} = 0$ $\prb$-almost surely. The second case
	contradicts (\ref{ineq:empdeg}). Hence, we have shown that $\nabla_\sigma
		F_\prb(\mean) > 0$ for all $\sigma \in \Sigma_\mean$, i.e. ($\nabla$).

\end{proof}

\subsection{Wasserstein Stickiness in \cat~Spaces}

The following theorem was initially stated and proven for the kale in
\cite[Theorem 7.6]{HMMN15}. We provide a
generalization to separable, locally compact and geodesically complete
\cat-spaces.

\begin{setting}
	\label{set:3}

	In addition to Scenario \ref{set:2}, assume that $(M,d)$ is \emph{geodesically complete} and \emph{locally compact}.
\end{setting}
%

Recall from Proposition \ref{prop:uniquefrech} that already Scenario \ref{set:2} ensures unique  Fr\'echet means.

\begin{theorem} 
	\label{theorem:mainthm}

	In Scenario \ref{set:3} let $q\geq 1$ 
	and $\prb \in \Qkapn[q]$.  Then, the following statements are equivalent.

	\begin{enumerate}
		\item[($W_q$)] $\prb$ is $(\Qkapn[q], W_{q})$-sticky at $\mean$,
		\item[(P)] $\prb$ is $\Bkapn$-perturbation sticky at $\mean$,
		\item[(Spc)] $\prb$ is sample sticky at $\mean$ and fulfills the pull condition (\ref{pullcond}),
		\item[($\nabla$)] $\prb$ is directionally sticky at $\mean$.
	\end{enumerate}
\end{theorem}

\begin{proof}

	``($W_q$) $\Rightarrow$ (P)'' follows from Theorem
	\ref{theorem:immediateflvrs} (iii). 
	Both ``(P) $\Rightarrow$ ($\nabla$)'' and ``(Spc) $\Rightarrow$ ($\nabla$)'' were shown in Theorem\ref{theorem:mainthmdir}.

	``($\nabla$) $\Rightarrow$ ($W_q$)'':
	For any measure $\prbalt \in \Qkapn$, due to Theorem \ref{theorem:derivcontr_sphere}, there is $C$, such that
	\begin{equation*}
		\sup_{\sigma \in \Sigma_{\mean}} \Big \lvert \nabla_\sigma  F_{\prb}
		(\mean) - \nabla_\sigma F_{\prbalt} (\mean) \Big \rvert \leq C \cdot
		W_1(\prb,\prbalt) \leq C \cdot W_{q}(\prb, \prbalt)
	\end{equation*}
	where we used $W_1 \leq W_{q}$ for 	$q \geq 1$, cf.\cite[Remark 6.1]{V08}. Since $\Sigma_\mu \to (0,\infty), \sigma \mapsto \nabla_\sigma F_\prb(\mean)$, by assumption, is continuous (Corollary \ref{corollary:contfrech}) and $\Sigma_{\mean}$ is compact (Proposition \ref{prop:cat_dir}), there is a lower bound $c > 0$ s.t.
	\begin{align*}
		0 < c = \min_{\sigma \in \Sigma_\mean} \nabla_\sigma F_\prb(\mean)
	\end{align*}
	for all $\sigma \in \Sigma_{\mean}$. In consequence, for all $\prbalt \in \Qkapn[q]$ with  $W_{q}(\prb,\prbalt) < \frac{c}{C}$, we have
	\begin{align*}
		\sup_{\sigma \in \Sigma_{\mean}} \Big\lvert \nabla_\sigma F_{\prb}
		(\mean) - \nabla_\sigma F_{\prbalt} (\mean) \Big \rvert \leq
		C \cdot W_{q^\prime}(\prb,\prbalt) < c\,,
	\end{align*}
	yielding that $\inf_{\sigma \in \Sigma_{\mean}} \nabla_\sigma F_{\prbalt} (\mean) > 0.$ Since every $Q\in \Qkapn[q]$ has, by Theorem \ref{theorem:meanderivs}, only one local minimizer in $\Bkap$, which is its Fr\'echet mean, we have that $\fb{(\prbalt}) = \{\mean\}$ for all $\prbalt \in \Qkapn[q]$ with  $W_{q}(\prb,\prbalt) < \frac{c}{C}$, thus establishing (W$_q$).

	``($\nabla$) $\Rightarrow$ (Spc)'': Above, we have seen that (D) $\Rightarrow$ (W$_q$) and in Theorem \ref{theorem:immediateflvrs} that (W$_q$) $\Rightarrow$ (P). Moreover, due to ($\nabla$), all directional derivatives are positive, so that there can be no direction with $\phi_{\mean,\sigma} = 0$ $\prb$-almost surely, yielding (Spc).

\end{proof}

\begin{remark}
	\begin{enumerate}
		\item The assumption of geodesic completeness could be relaxed to hold only locally at the Fr\'echet mean, making Theorem \ref{theorem:mainthm} also applicable to bounded CAT(0)-spaces like, say, a tripod, cf. Example \ref{ex:kspider}, with truncated legs.
		\item In consequence of Theorems \ref{theorem:immediateflvrs} and \ref{theorem:mainthm}, a
		      distribution in $\Qkapn$ featuring a $1\leq q$-th moment is
		      W$_q$-sticky if and only it is W$_1$-sticky.
	\end{enumerate}
\end{remark}

\begin{corollary} 
	\label{cor:stickytangent}

	In Scenario \ref{set:3}, a distribution $\prb \in \Qkapn[q]$ is $(\Qkapn[q], W_1)$ sticky at $\mean \in \Bkap$ if and only
	if the distribution $\log_{\mean}\# \prb$ is $(\wst[q]{T_\mean}, W_q)$-sticky at $\cO$.
\end{corollary}

%
%

\begin{proof}

	By Theorem \ref{thm:directional-sticky-tangent-cone}, we know that $\prb$ is directionally 	sticky if $\log_{\mean}\# \prb$ is so. Under the assumptions of Scenario \ref{set:3}, $T_x$ also fulfills the assumptions of Scenario \ref{set:3} as a CAT(0)-space due to Theorem \ref{prop:cat_dir}. Thus, Theorem \ref{theorem:mainthm} is applicable to both $\prb$ and $\log_{\mean}\# \prb$,
	yielding the claim.


\end{proof}

In this section, thus far, we have considered topological stickiness only with respect to the topology induced by the Wasserstein distance. Curiously, in principle, topological stickiness with respect to the topologies induced by an $f$-divergence with $f$ satisfying (fc) from the introduction can only occur if the space $(M,d)$ is bounded. If we have an unbounded CAT(0)-space, there cannot be any $f$-divergence sticky distributions as is shown in Theorem \ref{theorem:fdiver_unbounded} in the next Section \ref{scn:rule-out}.

\begin{theorem} 
	\label{thm:fdiver_bounded}
	In Scenario \ref{set:3} assume $\diam (\Bkap) < \infty$.
	Then, the following are equivalent for $\prb \in \Qkap$ with $\fb(\prb) = \mean$.
	\begin{enumerate}
		\item[(W)] $\prb$ is $(\Qkapn[q], W_q)$-sticky for any $q\geq 1$
		\item[(TV)] $\prb$ is $(\Qkap, \TV)$-sticky
		\item[($D_f$fc)] $\prb$ is $(\Qkap, D_f)$-sticky if $f$ is twice-differentiable
		      and $f^{\prime\prime}(x) >0$ for all $x \neq 1$.
	\end{enumerate}
\end{theorem}

\begin{proof}

	From Theorem \ref{theorem:immediateflvrs} we have (W$_1$) $\Rightarrow$
	(TV) $\Rightarrow$ ($D_f$fc). Again by
	Theorem \ref{theorem:immediateflvrs}, ($D_f$fc) implies $\Bkap$-perturbation stickiness.
	This, in turn, implies $(\Qkap, W_1)$-stickiness by Theorem \ref{theorem:mainthm}.

\end{proof}

\subsection{Ruling Out Stickiness}
\label{scn:rule-out}

We begin by showing that directional stickiness, and, thus, all other
flavors are not possible if $(M,d)$ is a sufficiently regular Riemannian manifold.

\begin{theorem}
	\label{thm:rulingoutmfd}

	Suppose $(M,d)$ is a complete connected Riemannian manifold of finite dimension, equipped with
	its intrinsic distance, and has upper sectional curvature bound $\kappa \in
		\RR$. Then for any $x_0 \in M$, there is no distribution $\prb \in \Qkapn$
	that is directionally sticky and, consequently, $(\Qkapn, W_1)$-sticky or $\Bkapn$-perturbation sticky.
	Furthermore, $\prb \in \Qkap$ is sample sticky if and only if it is a
	Dirac-measure.
\end{theorem}

\begin{proof}

	In case of differentiability it is well-known that the gradient vanishes at the Fr\'echet mean for a
	distribution $\prb \in \Qkapn$ (see e.g. \cite[Section 7]{kendall_mean}). Thus, all directional derivatives of
	$F_\prb$ vanish as well. By the assumptions above, we are in Scenario \ref{set:3} and Theorem \ref{theorem:mainthm} rules
	out the remaining flavors also.\\
	Let $\prb \in \Qkapn$ be sample sticky at $\mean \in \Bkapn$ and suppose it
	is not a Dirac-measure. Then there must be a
	direction $\sigma \in \Sigma_\mean$ such that $\prb(\phi_{\mean_\sigma}^{-1}(\{0\})) < 1$.
	However, as we saw above that $\nabla_\sigma F_\prb(\mean) =0$, this leads to
	a contradiction by the argument of the direction "$(\text{Spc}) \implies (\nabla)$" in Theorem \ref{theorem:mainthmdir}, recalling that Scenario \ref{set:2} holds.

\end{proof}

To also rule out stickiness at a point in a more general space, one can resort to the tangent cone at that point, since
Corollary \ref{cor:stickytangent} shows that stickiness of a probability measure is inherently tied to the tangent cone of the Fr\'echet mean. In Theorem \ref{theorem:rulingout}, we will see that the structure of the tangent cone at a point rules out 
potential stickiness at that point.

\begin{definition}
	Let $(M,d)$ be a \cat-space and take $x \in M$. The \emph{shadow} of a
	direction $\sigma \in \Sigma_x$ is given by the set
	\begin{align*}
		\mathfrak{S}_\sigma = \mathfrak{S}_{x, \sigma} = \{ \tau \in \Sigma_x \
		\vert \ \angle_x(\sigma, \tau) = \pi \}
	\end{align*}
	We call a shadow \emph{non-trivial} if $\lvert \mathfrak{S}_{\sigma} \rvert > 1$.
	The point $x$ is called \emph{prismatic} if all directions have a non-trivial shadow.
\end{definition}

\begin{figure}[!h]
	\label{fig:prism}

	\centering
	\includegraphics[trim={0 2cm 0 2cm},clip]{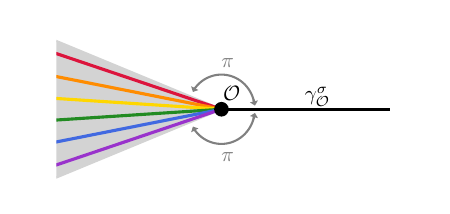}
	\caption{\it The cone point in a kale is prismatic, cf. Fig. \ref{fig:kalegeod} Each colored half line (in the shadow of the black half line) is a valid
		continuation of $\gamma_\mathcal{O}^\sigma$ resulting in a geodesic path, i.e. the
		corresponding directions lie in $\mathfrak{S}_{x, \sigma}$.}
\end{figure}

For the proof of Theorem \ref{theorem:rulingout} we need the following lemma, which is related to \cite[Corollary 4.4]{buildings}.

\begin{lemma} 
	\label{lemma:trivialshadow}

	Assume that $(M,d)$ is a geodesically complete and locally compact \cat~space, $\kappa \in \RR$, and that there is a point $x\in M$ with direction
	$\sigma \in \Sigma_x$ such that $\lvert \mathfrak{S}_{x,\sigma}\rvert = 1$.
	Then, there exists a Hadamard space $(M^\prime, d^\prime)$ with isometry
	\begin{align*}
		T_x\cong M^\prime \times \RR.
	\end{align*}
\end{lemma}

\begin{proof}
	For $\sigma \in \Sigma_x$ such that $\lvert \mathfrak{S}_\sigma\rvert = 1$  and $\tau \in \Sigma_x$, the only direction with $\angle_x(\sigma, \tau) = \pi$, define the geodesic $\gamma: \mathbb{R} \to T_x$ by
	\begin{align*}
		\gamma(t) = \left\{
		\begin{array}{ll}
			(\sigma, t) & \quad \text{if } t \geq 0, \\
			(\tau, -t)  & \quad \text{if } t < 0.    \\
		\end{array}
		\right.
	\end{align*}
	The assertion of the lemma then follows, if we show that every $y \in T_x \setminus\gamma(\RR)$ lies on a geodesic, parallel to $\gamma$, since then the product decomposition theorem for
	Hadamard spaces \cite[Theorem II.2.14]{BH11}, (which is applicable recalling that Proposition \ref{prop:cat_dir} guarantees that under our assumptions $T_x$ is a Hadamard space) asserts that there is a Hadamard space $M'$ such that $M$ is isometric to $M'\times \gamma(\RR)$.

	Hence, we now fix an arbitrary point $y \in T_x \setminus\gamma(\RR)$ and show that is lies on a geodesic with constant distance to $\gamma$.

	By our assumption, $\angle_\mathcal{O}(y,		\sigma) < \pi$. Since $\Sigma_x$ is a CAT(1)-space by Proposition \ref{prop:cat_dir} and, consequently, $\pi$-geodesic, and also geodesically complete, we have a minimizing geodesic $\alpha : [0, \pi] \to \Sigma_x$ with $\alpha(0) = \sigma$, $\alpha(\angle_\mathcal{O}(\sigma, y)) = \dir_\mathcal{O}(y)$ and $\alpha(\pi) = \tau$.

	\paragraph{First case, assume that $\angle_\mathcal{O}(\sigma, y) \neq \pi / 2$.} Then, let $r_y := \widetilde{d}(\mathcal{O}, y)$, where $\widetilde{d}$ denotes the distance on $T_x$ and set
	\begin{align*}
		y^\prime := \big (\alpha(\pi - \angle_\mathcal{O}(y,
				\sigma)), r_y \big ) \in T_x\setminus\gamma(\RR)\,.
	\end{align*}
	Set $l := d(y, y^\prime)/2$.
	Then,
	$\beta_y: \RR \to T_x$ with
	\begin{align*}
		\beta_y(t) = (\theta(t), r(t)) \quad t \in \RR\,,
	\end{align*}
	where $r(t) = \sqrt{t^2+ r_y^2 - l^2}$,
	$\theta(t) = \alpha(\arccos(t / r(t)))$ is a curve through $y=\left(\theta(l), r_y\right) =\beta(l)$ and $y'=\beta(-l)$, since
	\begin{align*}
		\arccos\left(\frac{l}{r_y}\right) & = \arccos\left(\frac{\widetilde{d}(y,y^\prime)}{2 r_y}\right)                                                           \\
		                                  & = \arccos\left(\frac{\sqrt{r_y^2 + r_y^2 - 2 r_y^2 \cos(\pi -2 \cdot \angle_\mathcal{O}(y, \sigma)))}}{2 r_y}\right)    \\
		                                  & = \arccos\left(\sqrt{\frac{1 + \cos(2 \cdot \angle_\mathcal{O}(y, \sigma))}{2}}\right)                                  \\
		                                  & = \arccos\left( \sqrt{\frac{2 + (e^{2i\angle_\mathcal{O}(y, \sigma)} + e^{2i\angle_\mathcal{O}(y, \sigma)})}{4}}\right) \\
		                                  & = \arccos\left( \sqrt{\frac{(e^{i\angle_\mathcal{O}(y, \sigma)} + e^{i\angle_\mathcal{O}(y, \sigma)})^2}{4}} \right)    \\
		                                  & = \arccos\left( \cos(\angle_\mathcal{O}(y, \sigma)) \right)                                                             \\
		                                  & = \angle_\mathcal{O}(y, \sigma) \,,
	\end{align*}
	where, we used that $\angle_\mathcal{O}(y,y^\prime) = \lvert \pi - 2 \cdot \angle_{O}(\sigma, y)\rvert$ and $\cos(u)=\cos(-u)$, $u \in \RR$, for the second equality and for the third that $\cos(\pi - u) = - \cos(u)$ for $u \in \RR$.


	Further, $\beta_y$ is a geodesic, since for arbitrary $t_1, t_2 \in \RR$ with
	$t_1 \leq t_2$, we have
	\begin{align*}
		\widetilde{d}^2(\beta_y(t_1), \beta_y(t_2)) & = r(t_1)^2 + r(t_2)^2 - 2 r(t_1)r(t_2) \cos\left(\angle_{\mathcal{O}}(\theta(t_1),\theta(t_2)\right)                                                    \\
		                                            & = r(t_1)^2 + r(t_2)^2 - 2 r(t_1)r(t_2) \cos\left(\arccos\left(\frac{t_2}{r(t_2)}\right) - \arccos\left(\frac{t_2}{r(t_2)}\right)\right)                 \\
		                                            & = r(t_1)^2 + r(t_2)^2 - 2 r(t_1)r(t_2) \left(\frac{t _1 t_2}{r(t_1)r(t_2)} - \sqrt{1 - \frac{t_1^2}{r(t_1)^2}}\sqrt{1 - \frac{t_2^2}{r(t_2)^2}})\right) \\
		                                            & = t_1^2 + t_2^2 + 2(r_y^2 - l^2) - 2 t_1 t_2 - 2 \sqrt{r(t_1)^2 - t_1^2}\sqrt{r(t_2)^2 - t_2^2}                                                         \\
		                                            & = (t_1 - t_2)^2 + 2(r_y^2 - l^2) - 2 (r_y^2 - l^2)                                                                                                      \\
		                                            & = (t_1 - t_2)^2\,,
	\end{align*}
	where the first equality follows from Definition \ref{Def:Cones}, the second, since  $\theta$ is a geodesic, and the third utilized $\cos(u-v) = \cos(u) \cos(v) - \sin(u) \sin(v)$ and 	$\sin(\arccos(w)) = \sqrt{1 - w^2}$ for $u, v,w \in \RR$ with $w^2\leq 1$.


	Finally, we show that $\beta_y$ has constant distance to $\gamma_y$. For arbitrary $s,t \in \RR$, due to Definition \ref{Def:Cones},  we have 
	\begin{align*}
		\widetilde{d}^2(\beta_y(t), \gamma(s)) & =
		s^2 + r(t)^2 - 2|s| r(t) \cos(\angle_\mathcal{O}(\beta(t), \gamma(s))                                                           \\
		                                       & =
		\begin{cases}
			s^2+t^2 + r_y^2 - l^2 + 2s r(t) \cos(\angle_\mathcal{O}(\tau, \theta(t))) \quad   & s < 0 \,, \\
			s^2+t^2 + r_y^2 - l^2 - 2s r(t) \cos(\angle_\mathcal{O}(\sigma, \theta(t))) \quad & s \geq 0  \\
		\end{cases}
		\\
		                                       & =
		\begin{cases}
			s^2+t^2 + r_y^2 - l^2 + 2s r(t) \cos(\pi - \angle_\mathcal{O}(\sigma, \theta(t))) \quad & s < 0 \,, \\
			s^2+t^2 + r_y^2 - l^2 - 2s r(t) \cos(\angle_\mathcal{O}(\sigma, \theta(t))) \quad       & s \geq 0  \\
		\end{cases}
		\\
		                                       & = s^2+t^2 +  r_y^2 - l^2 - 2s r(t) \cos\left(\arccos\left(\frac{t}{r(t)}\right)\right) \\
		                                       & = (s-t)^2 +r_y^2 - l^2\,.                                                              
	\end{align*}
	The minimum is attained for $s=t$ and this minimum is indeed constant over $t\in\RR$, completing the proof of the lemma in the first case.

	\paragraph{Second case, assume that $\angle_\mathcal{O}(\sigma, y) = \pi / 2$.}

	Then, consider the point $z = (\alpha(\pi/4), \sqrt{2} r_y)$. Construct the geodesic $\beta_z : \RR \to T_x$ and the point $z^\prime$ as before.
	We obtain
	\begin{align*}
		\beta_z(0) & = \left(\alpha(\pi/2), \sqrt{d^2(\mathcal{O}, z) - \frac{d^2(z, z^\prime)}{4}} \right) = \left(\dir_\cO (y), \sqrt{d^2(\cO, z) - \frac{2 d^2(\cO, z) - 2 d^2(\cO, z) \cos(\pi/2)}{4}}\right) \\
		           & = \left(\dir_\cO (y), \frac{d(\cO, z)}{\sqrt{2}}\right) = (\dir_\cO, r_y) = y\,.
	\end{align*}

\end{proof}

\begin{theorem} 
	\label{theorem:rulingout}
	In Scenario \ref{set:3} (recall that $\mean \in \Bkapn$) there exists a distribution $\prb \in \Qkapn$ that is
	$(\Qkapn, W_1)$-sticky at $\mean \in M$ if and only if $\mean$ is prismatic.
\end{theorem}

\begin{proof}

	First, we show ``$\mean$ prismatic $\Rightarrow$ $\exists \prb$ $(\Qkapn,W_1)$ sticky''.

	Pick arbitrary $\tau \in \Sigma_\mean$ and $\sigma= \sigma_\tau \in
		\mathfrak{S}_{\mu,\tau}$. Then $\tau \in \mathfrak{S}_{\mu,\sigma}$ and by hypothesis,
	there is another $\tau^\prime \in \mathfrak{S}_{\mu,\sigma}$ with $\tau \neq
		\tau^\prime$.

	Let $R >0$ be small enough such that $\overline{B_R(\mean)} \subset \Bkap$.
	Pick $y, z, z^\prime$ with	
	\begin{align*}
		\dir_\mean(y)        & = \sigma,                               \\
		\dir_\mean(z)        & = \tau,                                 \\
		\dir_\mean(z^\prime) & = \tau^\prime,                          \\
		d(\mean, y)          & = d(\mean, z) = d(\mean, z^\prime) = R, \\
	\end{align*}
	giving rise to the probability measure
	\begin{align*}
		\prb_\tau = \frac{1}{2} \left(\frac{1}{2}\left(\delta_{y} +
			\delta_{z}\right) + \frac{1}{2}\left(\delta_{y}
			+ \delta_{z^\prime}\right) \right).
	\end{align*}

	The point $\mean$ is the geodesic midpoint of both $y$ and $z$
	and $y$ and $z^\prime$, respectively. Consequently, we have that $$\mean = \fb\left(\frac{1}{2}\left(\delta_{y} +
			\delta_{z}\right)\right) = \fb\left(\frac{1}{2}\left(\delta_{y} +
			\delta_{z^\prime}\right)\right).$$
	Thus, $\mean = \mathfrak{b}(\prb_\tau)$ as
	\begin{align*}
		\min_{v \in M} F_{\prb_\tau}(v) & = \min_{v \in M} \frac{1}{2} \left(F_{\frac{1}{2}\left(\delta_{y} +
				\delta_{z}\right)}(v) + F_{\frac{1}{2}\left(\delta_{y} +
		\delta_{z^\prime}\right)}(v)\right)                                                                       \\
		                                & \geq  \frac{1}{2} \left(\min_{v \in M} F_{\frac{1}{2}\left(\delta_{y} +
				\delta_{z}\right)}(v) +\min_{v \in M} F_{\frac{1}{2}\left(\delta_{y} +
				\delta_{z^\prime}\right)}(v)\right).
	\end{align*}
	Furthermore, as we assumed that $\angle_\mean(\tau, \tau^\prime) > 0$, observe, due to Theorem \ref{theorem:pull},
	\begin{align*}
		\nabla_{\tau} F_{\prb_\tau}(\mean)
		 & = -\frac{R}{4}\Big(\cos \angle_\mean(\tau,\sigma) +\cos \angle_\mean(\tau,\tau)+\cos \angle_\mean(\tau,\sigma)+\cos \angle_\mean(\tau,\tau')\Big) \\
		 & = \frac{R}{4}\Big(1 - \cos \angle_\mean(\tau,\tau')\Big) > 0.
	\end{align*}
	As $\Sigma_\mean \to \RR, \upsilon \mapsto \nabla_{\upsilon}F_{\prb_\tau}(\mean)$ is continuous (Corollary \ref{corollary:contfrech}), there is an open set
	$\tau \in U_\tau \subset \Sigma_\mean$, such that
	$$\nabla_{\upsilon}F_{\prb_\tau}(\mean) >0 \mbox{ for all }\upsilon \in U_\tau\,.$$
	Recalling that $\tau\in \Sigma_\mean$ was arbitrary, take the union over all $\tau\in \Sigma_\mean$, to  obtain an open cover,
	i.e.
	\begin{align*}
		\bigcup_{\tau \in \Sigma_\mean} U_{\tau} = \Sigma_\mean.
	\end{align*}
	But since the space of directions is compact by Proposition
	\ref{prop:cat_dir}, already a finite number of directions $\tau_1, \ldots,
		\tau_K$ yield a covering
	\begin{align*}
		\Sigma_\mean = \bigcup_{i = 1}^K U_{\tau_i}.
	\end{align*}
	Thus, the probability measure $\prb = \frac{1}{K} \sum_{i=1}^K
		\prb_{\tau_i}$, which is supported on $\overline{B_{R}(\mean)} \subset
		\Bkap$, has a unique mean that is $(\Qkapn, W_1)$-sticky by Theorem
	\ref{theorem:mainthm} as all directional derivatives are positive.

	%
	Now the other direction ``$\exists \prb$ sticky at $\mean$ $\Rightarrow \lvert \mathfrak{S}_{\mu, \sigma} \rvert > 1 $ for all $\sigma \in \Sigma_\mean$''.
	Let $\prb \in \wst{M}$ be $(\Qkapn, W_1)$-sticky sticky at $\mean$. Due to Corollary \ref{cor:stickytangent} and Theorem 	\ref{theorem:mainthm}, the pushforward measure $\widetilde{\prb} = \log_\mean \# \prb$ must be $T_\mean$-perturbation sticky at $\mathcal{O} \in	T_\mean$.

	We now show that there cannot be any direction $\sigma \in \Sigma_\mean$ such that there is only one single $\tau \in \mathfrak{S}_{\mu,\sigma}$, i.e. $\angle_\mean(\sigma, \tau) = \pi$. For otherwise, due to Lemma \ref{lemma:trivialshadow}, we know that
	$T_\mean \cong M^\prime \times \RR$ for some Hadamard space $(M^\prime,d^\prime)$. Let $\psi: T_\mean \to M^\prime \times \RR$ be such an isometry and
	w.l.o.g. $\psi(\mathcal{O}) =
		(\mathcal{O}^\prime,0)$. However, we have by
	\cite[Proposition 5.5]{Sturm} for any $t \in (0,1)$ that
	\begin{align*}
		\mathfrak{b}((1-t) \psi\#\widetilde{\prb} + t \cdot \delta_{(\mathcal{O},1)}) = (\mathcal{O},t).
	\end{align*}
	Therefore, $\psi\#\widetilde{\prb}$ is not $M^\prime \times \RR$-perturbation sticky at $\psi(\cO)$, and by Theorem
	$\ref{theorem:mainthm}$ also not $(\wst[1]{M^\prime \times \RR}, W_1)$-sticky.
	As $\psi \#: \wst{T_\mu} \to M^\prime \times \RR$ is an isometry \cite[Lemma 5.1]{Sturm}, this implies that
	$\widetilde{\prb}$ is not $(T_\mean, W_1)$-sticky, leading to a contradiction.
\end{proof}


Theorem \ref{thm:rulingoutmfd} and \ref{theorem:mainthmdir} show that stickiness
appears only in rather unwieldy spaces and thus, in part answers questions raised in \cite{HE_MFO_2014}.

%
%

This following result is a direct consequence of Theorem \ref{theorem:rulingout}.

\begin{corollary} 
	\label{cor:ruling-out-main}

	In Scenario \ref{set:3}, there are no distributions that are $(\Qkapn, W_1)$-sticky at $\mean\in \Bkapn$ if every shadow in $\Sigma_\mean$ is trivial.
\end{corollary}

Corollary \ref{cor:ruling-out-main} is a generalization of  the previous Theorem \ref{thm:rulingoutmfd}
on Riemannian manifolds, that can be
useful to rule out stickiness in more general manifolds.

%
%
%
%
%

\begin{corollary}\label{cor:nostickymfd-curvature-bound-from-below}
	In Scenario \ref{set:3} assume that $(M,d)$ has no topological boundary and features a lower curvature bound. Then there is no $(\Qkapn, W_1)$-sticky distribution for any $x_0 \in M$.
\end{corollary}

\begin{proof}


	Spaces without topological boundary, that are of curvature bounded from below and above are
	thrice-differentiable manifolds \cite[Theorem 10.10.13]{burago}. As we
	assume geodesic completeness, $(M,d)$ cannot have a boundary. As such,
	the space of directions in such a space is also a sphere.
\end{proof}

\begin{remark}

	As we have seen in our scenario that stickiness 
	is not possible on manifolds, it occurs, however, on stratified spaces, as we have also seen, like the BHV spaces and the other spaces introduced in Section \ref{scn:3-flavors}. 

	On a similar note, Theorem \ref{thm:rulingoutmfd} implies that stickiness cannot occur in the space of symmetric positive definite matrices equipped the usual Killing geometry \cite[Chapter XII]{lang}. However, 	it has been conjectured that its subspace, the \emph{wald space} of phylogenetic forests \cite{wald1,wald2} 	features stickiness.

\end{remark}

\begin{remark}
	\label{example:partlysticky}

	In the proof of Theorem \ref{theorem:rulingout}, the individual
	measures $\prb_\tau$ are examples of what has been coined \emph{partly
		sticky} in for example \cite{HHL+13} and \cite{HMMN15}, meaning only some
	directional derivatives of the Fr\'echet function at the Fr\'echet mean
	are positive. We have seen that $\nabla_\tau F_{P_\tau} > 0$, however
	$\nabla_\sigma F_{P_\tau} = R\cdot (\frac{1}{2} - \frac{1}{2}) = 0$.
\end{remark}

One might ask for which distributions $\prb$ the classical condition $\nabla_\sigma F_\prb (x) =
	0$ holds for all directions $\sigma \in \Sigma_\mean$. We state the following
conjecture and prove it for special cases.

\begin{conjecture} 

	In Scenario \ref{set:3} let $\mean \in M$ be prismatic. 	Then for $\prb \in \Qkap$n, one has $\nabla_\sigma F_\prb(\mean) = 0$ for all 	$\sigma \in \Sigma_\mean$ if and only if $\prb = \delta_\mean$.
\end{conjecture}

One direction is trivial: if $\prb= \delta_\mean$, then also $\nabla_\sigma
	\prb(\mean) = 0$ for all $\sigma \in \Sigma_\mean$. We shall provide a proof
for the other direction for the special case that $\Sigma_\mean$ consists of
at least two disconnected components.

Under the assumptions of Scenario \ref{set:3}, we have that the space of directions at a point $x$ is a
CAT(1) space and as such $\pi$-geodesic. However, it might still be a disconnected space as long as the pairwise distances between the connected components is $\pi$. The simplest example for such a space would be the space of directions at the origin of a $K$-spider, consisting of $K$ points with pairwise distance $\pi$. Recall that $\Sigma_x$ is compact in Scenario \ref{set:3}, due to Proposition \ref{prop:cat_dir}, 
and thus, may only consist of finitely many connected components.

\begin{lemma} 

	In Scenario \ref{set:3} let $\mean \in M$ be prismatic. Furthermore, assume that $\Sigma_\mean$ consists of pairwise disjoint connected components $\Sigma_\mean^1, \ldots, \Sigma_\mean^k$. Then for $\prb \in \Qkapn$, one has
	$\nabla_\sigma F_\prb(\mean) = 0$ for all $\sigma \in \Sigma_\mean$ if and
	only if $\prb = \delta_\mean$.
\end{lemma}

\begin{proof}

	Due to (\ref{eq:directional-derivs-tangent-cone}), 
	we can equivalently consider the pushforward $\log_\mu \sharp P$ of $\prb \in \Qkapn$ as well as the distance $\widetilde{d}$ on the tangent cone. For simplicity, we call both $\prb$ and $d$, respectively.
	Write $T_\mean^1, \ldots, T_\mean^k$ for the components of the tangent cone solely connected by
	paths over the cone point corresponding to $\Sigma_\mean^1, \ldots, \Sigma_\mean^k$, fix arbitrary directions $\sigma_i \in \Sigma_\mean^i$ for $1 \leq i \leq k$ and define
	\begin{align*}
		m_i              = \int_{T_\mean^i} d(\mathcal{O}, z) \ \diff \prb(z) \,, \quad
		\widetilde{m}_i  = \int_{T_\mean^i} d(\mathcal{O}, z) \cdot \cos(\angle_x(\sigma_i, z)) \ \diff \prb(z)\,.
	\end{align*}
	If	$\nabla_\sigma F_\prb(\mathcal{O}) = 0$ for all $\sigma \in \Sigma_\mean$, then
	\begin{align*}
		0 & = -\sum_{i=1}^k \nabla_{\sigma_i} F_\prb(\mathcal{O}) = \sum_{i=1}^k \left(\widetilde{m}_i - \sum_{\substack{j = 1 \\ j \neq i}}^k m_j \right)
		= \sum_{i=1}^k \widetilde{m}_i - (k-1) \sum_{i=1}^k m_i
	\end{align*}
	and, hence,
	\begin{align*}
		(k-2) \sum_{i=1}^k m_i = \sum_{i=1}^k (\widetilde{m}_i - m_i).
	\end{align*}

	Obviously $(k-2) \sum_{i=1}^k m_i \geq 0$, whereas $\sum_{i=1}^k
		(\widetilde{m}_i - m_i) \leq 0$ since $(\widetilde{m}_i - m_i) \leq 0$ for $1\leq i \leq k$. Thus, both sides must be equal to 0.

	If $k > 2$, it follows that $m_i = 0$ for all $i$ and hence $\prb = \delta_\mathcal{O}$.

	If $k = 2$, note that $\widetilde{m}_i = m_i$ for $i=1,k$. Suppose $\prb \neq \delta_\mathcal{O}$, then $m_1 \neq 0$ or $m_2 \neq 0$. If only one of them is $0$, say, $m_1 >  0 = m_2$, then 	$\widetilde{m}_2 =0$, too, and hence $0 = \nabla_{\sigma_2}
		F_\prb(\mathcal{O}) = m_1 > 0$, a contradiction.

	Hence, we have $m_1,m_2 > 0$. Then $\prb$-almost surely $\cos(\angle_\mathcal{O}(\sigma_i, \dir_\mathcal{O}(z))) = 1$ on $T_\mean^i$, which, since the choice of $\sigma_i$ was arbitrary, must hold for all $\sigma_i \in \Sigma_\mean^i$, $i=1,2$. This implies that each $\Sigma_\mean^1$ and $\Sigma_\mean^1$ can only contain a single element. As $k=2$, this contradicts the assumption that all directions have a non-trivial shadow. Therefore, $\prb = \delta_\mathcal{O}$ as asserted.

\end{proof}

\paragraph{No $D_f$ stickiness on unbounded \cat~spaces.}

Now comes the curious result, that on unbounded spaces there are no sticky $f$-divergence distributions for sufficiently regular $f$. To this end, for $\prb \in \wst{M}$, $0\leq t \leq 1$ and $y\in M$ recall from Definition \ref{def:sticky} the perturbed measure
$$ \prb_t^y = (1-t) \prb + t\delta_y.$$
In Scenario \ref{set:1} we always have that for every $\epsilon > 0$ and fixed $y\in M$ that there is
$t_{y,\epsilon} > 0$ such that $$ W_1(\prb,\prb_t^y) < \epsilon\mbox{ for all }0\leq t \leq t_{y,\epsilon}\,,$$
cf. (\ref{eq:Wasserstein-perturbtion}). In contrast, for regular $f$-divergences we
have the following.

\begin{lemma} 
	\label{lem:pertb-in-div-neighborhood}

	Suppose that $(M,d)$ is an unbounded geodesic metric space and let $\prb \in \wst{M}$, $0\leq t \leq 1$, $x \in M$ and $\epsilon > 0$. Then there is $t_\epsilon > 0$ and a sequence $y_1,y_2, \ldots \in M$ with $d(x,y_n) \to \infty$ such that
	$$D_f(\prb \Vert \prb^{y_n}_t) < \epsilon\mbox{ for all } 0\leq t < t_\epsilon\mbox{and }n\in \NN$$
	for all $f$ giving rise to an $f$-divergence as in $D_f$ Definition \ref{def:divergence}.

\end{lemma}

\begin{proof}

	Since every $\prb \in \mathcal{P}(M)$ has at most countably many point
	masses, there is a sequence $y_1,y_2, \ldots \in M$ with $d(x,y_n) \to
		\infty$ and $\prb\{y_n\} =0$ for all $n\in \NN$. Recalling Equation
	(\ref{eq:fdvirpert}),
	\begin{align*}
		D_f(\prb \Vert \prb^{y_n}_t) & = (1-t)	 \cdot f\left( \frac{1}{1- t}\right)  \mbox{ for all }n\in \NN
	\end{align*}
	which is independent of $y_n$,  the assertion follows at once.
\end{proof}

To be able to also consider distributions possibly supported on the whole \cat~space, even when $\kappa>0$, we shall
use the following result generalizing Theorem \ref{theorem:pull} and \ref{theorem:meanderivs}
for distributions that are possibly supported on the hole space.

\begin{theorem}
	\label{theorem:meanderivsII}

	Let $(M,d)$ be a \emph{geodesic} and \emph{locally compact} \cat~space, $\kappa \in \RR$ and let $\mean \in M$.
	Then $\nabla_\sigma F_\prb(\mean)$ is well-defined for all $\sigma \in \Sigma_\mean$ and $\prb \in \wst{M}$.
	Furthermore, if $\mean \in \fb(\prb)$, it holds
	\begin{align*}
		\nabla_\sigma F_\prb(\mean) \geq 0 \text{ for all } \sigma \in \Sigma_\mean\,.
	\end{align*}

\end{theorem}

\begin{proof}
	Due to \cite[Corollary 4.5.8]{BBI01}, we have that $\nabla_\sigma d_z(\mean)$ exists for every $\sigma \in \Sigma_\mean$.
	The remainder of the proof is as in Theorem \ref{theorem:pull} and \ref{theorem:meanderivs}.

\end{proof}

\begin{theorem} 
	\label{theorem:fdiver_unbounded}


	Suppose that $(M,d)$ is a locally compact and complete \emph{geodesic} \cat~space, $\kappa \in \RR$ and that $f$ gives rise to an $f$-divergence $D_f$ as in Definition \ref{def:divergence}.
	For $\mean \in M$, assume there is a minimizing geodesic $\gamma: [0,\infty) \to M$.
	Then, there is no distribution that is $(\wst{M},D_f)$-sticky on a closed $S \subset M$ with $\mean \in S$.

\end{theorem}

\begin{proof}

	Let $y\in M$, $0\leq t\leq 1$ and $\prb\in \wst{M}$ such that $\mean \in \fb(\prb)$.
	In both cases, let
	$\sigma$ be the direction of $\gamma$ at $\mean$. Then,
	\begin{align*}
		\nabla_\sigma F_{\prb_t^y} (\mean) & = (1-t) \cdot\nabla_\sigma F_\prb(\mean) - t \cdot d(\mu,y)
	\end{align*}
	which is negative, if $d(\mu,y) > \frac{1-t}{t}\,\nabla_\sigma F_\prb(\mean)$, i.e. in that case $\mu \not\in \fb(\prb_t^y)$, due to Theorem \ref{theorem:meanderivsII}.
	Due to Lemma \ref{lem:pertb-in-div-neighborhood}, however, for every $\epsilon > 0$ and $t>0$ there is $y \in \gamma([0,\infty))$ such that $D_f(\prb \Vert \prb^{y}_t) < \epsilon$. Thus, $\prb$ is not $(\wst{M},D_f)$-sticky.
\end{proof}

\begin{remark}

	Theorem \ref{theorem:fdiver_unbounded} covers many classical Hadamard spaces such as BHV-spaces and $K$-spiders.\\
	It, however, does not cover all unbounded Hadamard spaces.
	Consider the sets $\{ [0,n] \ : \ n \in NN\}$ and glue them together at $0$, i.e. a "$\NN$-spider" with truncated legs.
	This space is clearly unbounded but neither locally compact nor is there a minimizing geodesic defined on $[0,\infty)$.
\end{remark}

\section{Modulation Stickiness and Its Asymptotics}
\label{scn:mod-sticky-asymp}

In this section, we will show some asymptotic results involving the directional derivatives using empirical process theory. Let us introduce some basic notions from
\cite{vdVaart}.

\begin{definition}

	The covering number $N(\epsilon, M, d)$ of a metric space $(M, d)$ is given by the minimal number of balls with radius $\epsilon > 0$ with
	respect to the metric $d$ such that their union covers the whole space.

\end{definition}

In the following, we will often require an assumption on the covering numbers of the space of directions, essentially requiring it to be finite dimensional of some dimension $k >0$.

\begin{setting}
	\label{set:4}

	In addition to Scenario \ref{set:3}, assume that there
	are $k > 0$ and $V > 0$ such that $N(\epsilon,
		\Sigma_y, \angle_y) \leq \frac{V} {\epsilon^k}$ for every $y \in \Bkapn$.
\end{setting}

%

\subsection{Modulation Stickiness}\label{scn:mod-sticky}

A useful tool in studying Fr\'echet means of distributions exhibiting non-standard limiting behavior is 
the \emph{variance modulation} from \cite{pennec2019curvature}, given for $\prb \in \wst[2]{M}$ by
\begin{align*}
	\mathfrak{m}_n^2(\prb) = \frac{n \cdot \EE(d^2(\mean, \smean))}{\EE(d^2(\mean, X_1))},
\end{align*}
for a sample 
$X_1, \ldots, X_n \iid \prb$.
With the population and sample means $\mu$ and $\hat \mu_n$, respectively, which are unique, already in Scenario \ref{set:2},  consider for $q \geq 1$ the more general  \emph{$q$-th moments modulation}:
\begin{align*}
	\mathfrak{m}_n^q(\prb) = \frac{n^{\frac{q}{2}} \cdot \EE(d^q(\mean, \smean))}{\EE(d^q(\mean, X_1))},
\end{align*}

We define a new flavor of stickiness through the limiting behavior of the $q$-th moments modulation under the pushforward of the log map at the population Fr\'echet mean.

\begin{definition}[Modulation Stickiness, Scenario \ref{set:3}]
	\label{def:modsticky}

	The distribution $\prb
		\in \Qkap[q]$ is called \emph{$q$-modulation sticky} at $\mean$, for $q \geq 1$, if for the
	pushforward measure $\widetilde{\prb} = \log_\mean \# \prb$ on $T_\mean$
	\begin{align*}
		\lim_{n \to \infty} \mathfrak{m}_n^q(\widetilde{\prb}) = 0.
	\end{align*}
\end{definition}

We have defined modulation stickiness on the tangent cone because it is there that the fundamental properties are rooted. For this reason we begin with the following theorem. Its elaborate proof is deferred to Section \ref{scn:mod-sticky-asymp}, starting on Page \pageref{proof:theorem:modunbounded}.  For $a, b \in \RR$, write $a \lor b= \max\{a,b\}$.

\begin{theorem} 
	\label{theorem:modunbounded}

	Let $(M,d)$ be a \catzero~space and a Euclidean cone over its space of directions $\Sigma_\cO$ at its cone point $\cO\in M$. Furthermore, assume that $\Sigma_\cO$ is compact as well as that there
	are $k > 0$ and $V > 0$ such that $N(\epsilon, \Sigma_\cO, \angle_\cO) \leq \frac{V} {\epsilon^k}$,
	Then, for $q \geq 1$, the following statements hold true.
	\begin{enumerate}
		\item If $\prb \in \wst[2\lor (q+\epsilon)]{M}$, $\epsilon >0$, is directionally sticky at
		      $\mean = \mathcal{O}$, we have
		      \begin{align*}
			      n^\frac{q}{2} \EE \left( d^q(\mean, \smean) \right) \to 0.
		      \end{align*}
		\item  If $\prb \in \wst[q \lor 2]{M}$ fulfills the pull condition \eqref{pullcond} and
		      \begin{align*}
			      \lim_{n \to \infty} n^\frac{q}{2} \cdot \EE \left( d^q(\mean, \smean)\right) = 0,
		      \end{align*}
		      then $\prb$ is directionally sticky at $\mean = \mathcal{O}$.
	\end{enumerate}

\end{theorem}

Next, we extend Part 1. of Theorem \ref{theorem:modunbounded} to CAT(0)-spaces that are not cones but sufficiently regular. Since $d$ is a convex function, there can at most be countably many points where the derivative of $d^2_z \circ \gamma$ is not defined. However, both left- and right derivatives exist for any $t \in [0, d(\mean,\meanalt)]$, and the left derivative is at most as large as the right derivative \cite[Chapter 4.4]{burago}. We assume below
that every geodesic segment features only finitely many points of discontinuity of the directional derivatives along it.

\begin{theorem}

	\label{thm:modunbounded}

	In Scenario \ref{set:4} let $(M,d)$ be a Hadamard space such that  for every geodesic segment $\gamma: [0,a] \to M$, $a>0$, $\prb \in \wst{M}$ and $z \in M$, the function $F_\prb \circ \gamma$ is twice-differentiable everywhere except for at most finitely points. 
	Then, for $\prb \in \wst[(q + \epsilon) \lor 2]{M}$, $q \geq 1$ and $\epsilon >0$, with Fr\'echet mean $\mean$, that is  directionally sticky, we have
	\begin{align*}
		n^\frac{q}{2} \EE \left( d^q(\mean, \smean) \right) \to 0.
	\end{align*}
\end{theorem}

\begin{proof}

	By Lemma \ref{lemma:tanmeandist} further below, we have $d(\mean, \smean) \leq \widetilde{d}
		(\mathcal{O}, \mathfrak{b}(\log_\mean \# \prb_n))$. Furthermore, the
	sequence $(\log_\mean(X_i))_{i\in \NN}$ is an i.i.d. sequence of random
	variables following the distribution $\log_\mean \# \prb$. Therefore, we
	can apply Theorem \ref{theorem:modunbounded}. Thus,
	\begin{align*}
		n^\frac{q}{2} \EE \left( d^q(\mean, \smean) \right) \leq  n^\frac{q}{2} \EE \left( d^q(\cO, \fb((\log_\mean \# \prb_n)) \right)\to 0.
	\end{align*}
\end{proof}

By definition of modulation stickiness, Theorem \ref{theorem:modunbounded} translates at once to the general Scenario \ref{set:4} as follows.

\begin{theorem} 
	\label{theorem:modstickiness}

	In Scenario \ref{set:4} the following two statements hold for $\prb \in \wst[2 \lor q]{M}$,
	$q \geq 1$.
	\begin{enumerate}
		\item If $q>1$ and if $\prb$ is directionally sticky 
		      then it is $p$-modulation sticky
		      for all $1\leq  p < q$.
		\item If $\prb$ is $q$-modulation sticky and fulfills \eqref{pullcond}, then
		      $\prb$ is also directionally sticky.
	\end{enumerate}
\end{theorem}




If we assume that our distribution $\prb$ is supported by a bounded subset $A \subset M$, we have, in case of stickiness, faster convergence rates.

\begin{theorem}[Distributions with bounded support
	]
	\label{theorem:modbounded}

	Suppose that in Scenario \ref{set:4} -- recall that $k$ gives the covering rate of the space of directions --  a distribution $\prb \in \Qkapn$ is directionally sticky at $\mean$ with
	$\supp(\prb) \subset A$ for a closed convex $A \subset M$ with $\diam(A) < \infty$. Then
	for every sample, $X_1, \ldots, X_n \iid \prb$, $n\in \NN$, there are constants $c_{min} =\min_{\sigma \in 			\Sigma_\mean} \nabla_\sigma F_\prb (\mean) > 0$ and $C >0$ that
	\begin{align*}
		\mathbb{P} \left( \mathfrak{b}(\prb_n) \neq \mean \right)  \leq C \left(\sqrt{n}c_{min}\right)^k \exp\left(- 2 n c_{min}^2 \right)
	\end{align*}
	where $C$ solely depends on $k$ and $V$. In particular, we have for every $q \geq 1$ that
	\begin{align*}
		\EE \left( d^q (\mean, \smean) \right) \leq C \cdot \diam(A)^q \cdot \left(\sqrt{n}c_{min}\right)^k \exp\left(- 2 n c_{min}^2 \right).
	\end{align*}
\end{theorem}

\begin{proof}

	Since we have that $\lvert \phi_{x,\sigma}(z) \rvert \leq \diam(A)$ for $z$ $\prb$-almost surely, and $N(\epsilon, \Sigma_\mean, \angle_\mean)
		\lesssim \frac{V}{\epsilon^q}$, applying \cite[Theorem 2.14.9]{vdVaart} to
	Inequality (\ref{ineq:empmeanprob}) yields
	\begin{align*}
		\mathbb{P} \left( \smean \neq \mean \right) & \leq \mathbb{P} \left( \sup_{\sigma \in \Sigma_\mean}\lvert \nabla_\sigma
		F_\prb(\mean) - \nabla_\sigma F_{\prb_n}(\mean)\rvert > c_{min}\right)                                                        \\
		                                            & \leq (\sqrt{n}\cdot c_{min})^k \cdot C \cdot \exp\left(- 2 n c_{min}^2 \right),
	\end{align*}
	where $C$ depends only on $V$ and $k$, yielding the first claim.

	As the $\supp(\prb_n) \subset A$, almost surely, both the population and sample Fr\'echet mean reside in $A$, the latter a.s., see \cite[Proposition 6.1]{Sturm} and \cite[Corollary 16]{yokota}. Thus,
	\begin{align*}
		\mathbb{E} \left(d^q(\mean, \smean)\right) & = \mathbb{E} \left( \mathbbm{1}_{\{\smean \neq \mean\}} \cdot d^q(\mean, \smean)\right) \\
		                                           & \leq \mathbb{P} \left( \smean \neq \mean \right) \cdot \diam(A)^q\,,
	\end{align*}
	and the second claim follows.
\end{proof}

Before embarking on proving the technical details of Theorem \ref{theorem:modunbounded} we first introduce some basic empirical process jargon.
\subsection{Some Empirical Process Terminology}

The rest of the notation introduced here will be
standard notation of empirical process theory. It will only be used in the subsequent proofs.

To derive asymptotic results, we will need to \emph{uniformly} control the convergence of the directional derivatives of the Fr\'echet function for samples. We know that $\nabla_\sigma F_\prb (\mean) =
	\int_M \phi_{\mean,\sigma}(z) \ \diff\prb(z)$ for any $\prb \in \wst{M}$.
Following standard notation in empirical process theory, we write for an integrable function $\psi: M \to \mathbb{R}$ and some 
measure $\prbalt$ on $M$,
\begin{align*}
	\prbalt \psi = \int_M \psi(z) \ \diff \prbalt(z).
\end{align*}
In this notation, we have $\nabla_\sigma F_\prb (\mean) = -\prb \phi_{\mean, \sigma}$.
Therefore, we can interpret the negative directional derivatives as integrals over the
set of pulls, indexed by the space of directions, i.e. the function class
\begin{align*}
	\Phi_{\Sigma_\mean} = \{ \phi_{\mean,\sigma} : M \to \mathbb{R} \ \vert \
	\sigma \in \Sigma_\mean\}
\end{align*}
for $\mean \in M$ in a \cat-space $(M,d)$. 

The \emph{centered and scaled empirical process} $\mathbb{G}_n$ is given by
\begin{align}\label{eq:centered-scaled-emp-process}
	\mathbb{G}_n = \sqrt{n} (\prb_n - \prb)\,,
\end{align}
and its supremum
\begin{align}\label{eq:sup-emp-process}
	\lVert \mathbb{G}_n \rVert_{\Phi_{\Sigma_\mean}} := \sup_{\psi \in \Phi_{\Sigma_\mean}} \lvert \mathbb{G}_n \psi \rvert\,.
\end{align}

We are interested in the \emph{$\Phi_{\Sigma_\mean}$-indexed empirical process}, given by
\begin{align*}
	\{ \mathbb{G}_n  \phi : \phi \in \Phi_{\Sigma_\mean} \}.
\end{align*}
Typically, in empirical process theory, one to verifies for a distribution $\prb \in \Qkap$ and  every $y \in M$ that
$$\sup_{\psi \in
		\Phi_{\Sigma_\mean}} \lvert \psi(y) - \prb \psi \rvert < \infty\,.$$
This is indeed the case in Scenario \ref{set:3} since
\begin{align*}
	\lvert \phi_{\mean, \sigma}(y) - \prb \phi_{\mean, \sigma} \rvert & = \lvert \phi_{\mean, \sigma}(y) - \nabla_\sigma F_\prb(\mean)\rvert           \\
	                                                                  & = \lvert \nabla_\sigma F_{\delta_y}(\mean) - \nabla_\sigma F_\prb(\mean)\rvert \\
	                                                                  & \leq C \cdot W_1(\delta_y, \prb)                                               \\
	                                                                  & = C \cdot \int_M d(y, z) \ \diff\prb(z) < \infty\,,
\end{align*}
for all $\sigma \in \Sigma_\mean$, where Theorem \ref{theorem:derivcontr_sphere} provides suitable $C>0$, and the assumption $\prb \in \wst{M}$ ensures that the integral at the end is finite.


\subsection{Details on Modulation Convergence
} \label{scn:moment-convergence}

Modulation stickiness hinges on the rate of convergence of the expected distance between sample and population mean. Given a distribution $\prb 	\in \wst[q]{M}$ with unique Fr\'echet mean $\mu$ and a sequence $X_1, X_2, \ldots \iid \prb$ with unique sample mean $\hat\mu_n$, we have
\begin{align*}
	\mathbb{E} \left(d^q(\mean, \smean)\right) & = \mathbb{E} \left( \mathbbm{1}_{\{\smean \neq \mean\}} \cdot d^q(\mean, \smean)\right),
\end{align*}
allowing both sides to be infinite as of now. We shall use this to obtain convergence rates that are faster than $n^{\frac{q}{2}}$ for the term on the left-hand sight by obtaining bounds of $\mathbb{P} \left(
	\mathfrak{b}(\prb_n) \neq \mean \right)$ in the case of sticky distributions.

As we saw in the proof of Theorem \ref{theorem:mainthm}, for a $W_1$-sticky distribution $\prb$ in Scenario \ref{set:3}, we can find a lower bound $c_{min} = \min_{\sigma \in \Sigma_\mean} \nabla_\sigma F_\prb (\mean) > 0$ on the directional derivatives of the Fr\'echet function. By Theorem \ref{theorem:meanderivs}, if $\mathfrak{b}(\prb_n) \neq \mean$, a directional derivative of the empirical measure's Fr\'echet function must be  negative. This then implies that $\sup_{\sigma \in \Sigma_\mean}\lvert \nabla_\sigma 	F_\prb(\mean) - \nabla_\sigma F_{\prb_n}(\mean)\rvert > c_{min}$ for some $c_{min}>0$, giving rise to
\begin{align}
	\label{ineq:empmeanprob}
	\mathbb{P} \left( \mathfrak{b}(\prb_n) \neq \mean \right) \leq \mathbb{P} \left( \sup_{\sigma \in \Sigma_\mean}\lvert \nabla_\sigma
	F_\prb(\mean) - \nabla_\sigma F_{\prb_n}(\mean)\rvert > c_{min}\right).
\end{align}

Using empirical process theory, we obtain bounds in terms of $n$ on the right-hand side by translating the problem into the well-behaved tangent cone, as Corollary \ref{cor:stickytangent} ensures that stickiness of a distribution in the original space is equivalent to stickiness of its pushforward
under the log map to the tangent cone at the population mean. On this Euclidean cone, we will derive a closed expression for the position of Fr\'echet means.

The following result determines the position of a Fr\'echet mean in the tangent cone terms of the directional derivatives of the Fr\'echet function at the cone point.

\begin{lemma} 
	\label{lemma:conemean}

	Let $(M,d)$ be a \catzero~space and a Euclidean cone over its space of directions $\Sigma_\cO$ at its cone point $\cO\in M$. Assuming that $\Sigma_\cO$ is compact, $\prb \in \wst{M}$ and $\sigma
		= \argmin_{\tau \in \Sigma_\mathcal{O}} \nabla_\tau F_\prb
		(\mathcal{O})$, we have
	\begin{align*}
		\mathfrak{b}(\prb) = (\sigma, 0 \lor (- \nabla_\sigma F_\prb (\mathcal{O}))).
	\end{align*}

\end{lemma}

\begin{proof}

	Let $\tau \in \Sigma_\mathcal{O}$. $\gamma_\tau \in \tau$ and consider arbitrary $z = (\rho,d(\cO,z)) \in M$.   Since $M$ is a Euclidean cone, cf. Definition \ref{Def:Cones}, we have
	\begin{align*}
		d^2(z, \gamma_\tau(t)) = t^2 + d^2(\cO,z) - 2 d(\cO,z)\, t \cos \angle_\cO(\tau, \rho)\,,
	\end{align*}
	for all $t \geq 0$. Consequently, in conjunction with Theorem \ref{theorem:pull},
	\begin{align*}
		F_\prb(\gamma_\tau(t)) = \frac{1}{2}\int d^2(z, \gamma_\tau(t)) \,d\prb(z) = \frac{1}{2} t^2 +  \nabla_\tau F_\prb(\mathcal{O}) \cdot t + F_\prb(\mathcal{O}), \quad t \geq 0\,,
	\end{align*}
	which attains its minimum $F_\prb(\mathcal{O})$ or $-\nabla_\tau F_\prb(\mathcal{O})^2/2$, at $t_\tau =0 \lor (- \nabla_\tau F_\prb (\mathcal{O})))$. Since $\Sigma_\cO$ is compact by hypothesis, $\nabla_\tau F_\prb (\mathcal{O})$ varies continuously in $\tau\in \Sigma_\cO$, due to Corollary \ref{corollary:contfrech},  and the uniqueness of the Fr\'echet mean on \catzero~spaces, due to Proposition
	\ref{prop:uniquefrech}, taking into account that our geodesics have unit speed, the claim follows.

	%

\end{proof}

\paragraph{Proof of Theorem \ref{theorem:modunbounded}}
\label{proof:theorem:modunbounded}
We start by proving Part 1.\\
Since, due to Theorem \ref{theorem:mainthm}, all  directional derivatives of $F_\prb$ at $\mu$ are positive, we have
\begin{align*}
	0 \lor \Big(- \sqrt{n}\underbrace{ \inf_{\tau \in \Sigma_\mean} \nabla_\tau F_{\prb_n} (\mathcal{O})}_{\geq \inf_{\tau \in \Sigma_\mean} \nabla_\tau F_{\prb_n} (\mathcal{O})- \nabla_\tau F_\prb (\mathcal{O})} \Big)
	\leq
	- \sqrt{n}\inf_{\tau \in \Sigma_\mean} \left(\nabla_\tau F_{\prb_n} (\mathcal{O}) - \nabla_\tau F_\prb(\mathcal{O})\right).
\end{align*}
In consequence, we have due to Lemma \ref{lemma:conemean} and recalling (\ref{eq:centered-scaled-emp-process}) and (\ref{eq:sup-emp-process}), that
\begin{align*}
	n^\frac{q}{2} \EE \left( d^q(\mean, \smean) \right) & = \EE \left( \mathbbm{1}_{\{\smean \neq \mean\}} \cdot \Bigg \lvert 0 \lor \left(- \sqrt{n}\inf_{\tau \in \Sigma_\mean} \nabla_\tau F_{\prb_n} (\mathcal{O})\right)\Bigg \rvert^q \right)                                    \\
	                                                    & \leq  \EE \left( \mathbbm{1}_{\{\smean \neq \mean\}} \cdot \Bigg\lvert \left( -\sqrt{n} \inf_{\tau \in \Sigma_\mean} (\nabla_\tau F_{\prb_n} (\mathcal{O}) - \nabla_\tau F_{\prb} (\mathcal{O}))\right)\Bigg\rvert^q \right) \\
	                                                    & \leq  \EE \left( \mathbbm{1}_{\{\smean \neq \mean\}} \cdot \lVert \mathbb{G}_n\rVert_{\Phi_{\Sigma_\mean}}^q \right).                                                                                                        \\
	%
	%
	                                                    & \leq  \EE \left(                                                                                                                                                                                                             
	\lVert \mathbb{G}_n\rVert_{\Phi_{\Sigma_\mean}}^{q+\epsilon} \right)^{\frac{q}{q+\epsilon}}
	\cdot \mathbb{P} \left( \smean \neq \mean \right)^\frac{\epsilon}{q+\epsilon}
\end{align*}
where, for the last inequality, we have used H\"older's inequality with coefficients $(q+\epsilon )/ q$ and	$(q+\epsilon )/ \epsilon$.

From the maximal inequalities \cite[Theorem 2.14.2 and 2.14.5]{vdVaart}, we know that
$\EE \left(\lVert \mathbb{G}_n\rVert_{\Phi_{\Sigma_\mean}}^{q+\epsilon} \right)^{\frac{1}{q+\epsilon}}$
is bounded due to \cite[Theorem
	2.7.11]{vdVaart}, valid due to the assumption on the covering number of $\Sigma_\mean$,  and since $\prb \in \wst[2 \lor
		(q+ \epsilon)]{M}$ (since the envelope function is $z \mapsto d(\mean,
	z)$).

Applying Markov's inequality to Inequality
(\ref{ineq:empmeanprob}) yields for the second factor above
\begin{align*}
	\mathbb{P} \left( \smean \neq \mean \right) & \leq \mathbb{P} \left( \sup_{\sigma \in \Sigma_\mean}\lvert \nabla_\sigma
	F_\prb(\mean) - \nabla_\sigma F_{\prb_n}(\mean)\rvert > c_{min}\right)                                                  \\
	                                            & \leq \frac{
		\EE \left(\lVert \mathbb{G}_n\rVert_{\Phi_{\Sigma_\mean}}^{q+\epsilon} \right)
	}{(c_{min} \cdot \sqrt{n})^{q+\epsilon}}.
\end{align*}
Again, the maximal inequalities tell us that the numerator above
is bounded, over $n$. Putting all together, we obtain for $n\to \infty$, as asserted,
\begin{align*}
	n^\frac{q}{2} \EE \left( d^q(\mean, \smean) \right) & \leq \frac{
		\EE \left(\lVert \mathbb{G}_n\rVert_{\Phi_{\Sigma_\mean}}^{q+\epsilon} \right)
	}
	{(c_{min})^\epsilon} \cdot n^{-\frac{\epsilon}{2}} \to 0\,.
\end{align*}

Next, let us show Part 2.\\
Assume the contrary, i.e. that 	$0 = \lim_{n \to \infty} n^\frac{q}{2} \cdot \EE \left( d^q(\mean, \smean)\right)$ but that $\prb$ is not directionally sticky, i.e. we find a direction $\sigma \in \Sigma_\mean$ with $\nabla_\sigma F_\prb(\mean) \leq 0$ by Theorem \ref{theorem:mainthm}.

From Lemma \ref{lemma:conemean}, we know that
\begin{align*}
	\EE \left( d^q(\mean, \smean )\right) & = \EE \left( \big( 0\lor (- \inf_{\tau \in \Sigma_\mean} \nabla_\tau F_{\prb_n} (\mathcal{O}))\big)^q \right)                      \\             & \geq \EE \left( \big( 0\lor (- \nabla_\sigma F_{\prb_n} (\mathcal{O}))\big)^q \right)                      \\
	                                      & \geq \EE \left( \big( 0 \lor (- \nabla_\sigma F_{\prb_n} (\mathcal{O}) + \nabla_\sigma F_{\prb} (\mathcal{O}))\}\big)^q \right)\,,
\end{align*}
where the last inequality holds since $\nabla_\sigma F_\prb(\mean) \leq 0$.
Hence,
\begin{align*}
	0 & = \lim_{n \to \infty} n^\frac{q}{2} \cdot \EE \left( d^q(\mean, \smean)\right)                                                                                                                           \\
	  & \geq  \lim_{n \to \infty} \EE \left( \big\lvert(0 \lor \left(- \sqrt{n} \cdot \left( \nabla_\sigma F_{\prb_n} (\mathcal{O}) -  \nabla_\sigma F_{\prb} (\mathcal{O})\right)\right)\big \rvert ^q \right), \\
\end{align*}
implying convergence in the $q$-th mean and, consequently, weak convergence of the integrand $0 \lor \left(- \sqrt{n} \cdot \left( \nabla_\sigma F_{\prb_n} (\mathcal{O}) -  \nabla_\sigma F_{\prb} (\mathcal{O})\right)\right)$ against the point measure $\delta_0$. However, as the pull $\phi_{\mean, \sigma}$ is non-zero with positive probability, due to the pull condition \eqref{pullcond}, we have by the univariate central limit theorem that $\sqrt{n} ( \nabla_\sigma F_{\prb_n} 		(\mathcal{O})- \nabla_\sigma F_\prb (\mathcal{O}))$ converges weakly against a Gaussian distribution with non-zero variance. By the continuous 	mapping theorem, $0 \lor (- \sqrt{n} ( \nabla_\sigma F_{\prb_n} (\mathcal{O})- \nabla_\sigma F_\prb (\mathcal{O})))$ converges weakly 	against a mixture of a point mass $\delta_0$ and a truncated Gaussian distribution, leading to a contradiction.
\qed

The following lemma on which the proof of Theorem \ref{thm:modunbounded} is based, seems to be folklore. As the authors did not find it in the literature, here it is with a short proof.

\begin{lemma} 
	\label{lemma:tanmeandist}

	Let $(M,d)$ be a Hadamard space with compact spaces of directions $\Sigma_y$, $y\in M$, such that  for every geodesic segment $\gamma: [0,a] \to M$, $a>0$, $\prb \in \wst{M}$ and $z \in M$, the function $F_\prb
		\circ \gamma$ is twice-differentiable everywhere but finitely points. Then
	for any $\prb \in \wst{M}$ and $x \in M$, it holds that
	$$d(x,\mathfrak{b}(\prb)) \leq \widetilde{d}(\mathcal{O}, \mathfrak{b} (\widetilde{\prb}))\,,$$
	where $\widetilde{d}$ denotes the
	distance on $T_\mu$ and $\widetilde{\prb} = \log_\mu \# \prb$.

\end{lemma}

\begin{proof}
	Since means on Hadamard spaces are unique, due to Proposition \ref{prop:uniquefrech}, let $\{\mean\} = \mathfrak{b}(\prb)$. If $\mean = x$, then the assertion is trivially true. Thus, assume that $\mean \neq x$. Let $\gamma: [0, a] \to M$ denote the unique geodesic from $x=\gamma(0)$ to $\mean=\gamma(a)$  (in particular, $a = d(\mean, x)$) 	and let $\sigma = \dir_x (\mean)$. Let $0 < a_1 \leq a_2 \leq \ldots \leq a_K < a$ be the points where $F_\prb$ is not differentiable. We set $a_0 = 0$ and $a_{K+1} = a$

	Write $\ddot{F}_\prb := \frac{\diff^2}{\diff t^2} F_\prbalt \circ \gamma$ whenever defined, $\dot{F}_\prbalt^- := \partial_t^- {F}_\prbalt \circ \gamma$ and $\dot{F}_\prbalt^+ = \partial_t^+ {F}_\prbalt \circ \gamma$, with  the left and right derivatives $\partial_t^-$ and $\partial_t^+$, respectively. For any $z \in M$ and $t \in (a_i, a_{i+1})$, we have by the
	NPC-inequality (\ref{ineq:NPC}), taking into account that our geodesics are of unit speed, that the second derivative (computed as
	the symmetric version) of $d^2(z, \gamma(\cdot))$ at $t$ is bounded from
	below as follows:
	\begin{align*}
		\frac{\diff^2}{\diff t^2} d^2(z,\gamma(t)) & = \lim_{h \to 0} \frac{d^2(z, \gamma(t+h)) + d^2(z,\gamma(t-h)) - 2 d^2(z,\gamma(t)) }{h^2} \\                            & \geq \lim_{h \to 0}  \frac{\frac{1}{2}d^2(\gamma(t+h), \gamma(t-h))}{h^2}
		\\
		                                           & = \lim_{h \to 0}  \frac{\frac{4}{2} h^2}{h^2}
		= 2\,,
	\end{align*}
	yielding $\ddot{F}_\prb\geq 1$. Hence,
	\begin{align*}
		d(x,\mean) = a = \int_0^a 1 \ \diff t & \leq \sum_{i=0}^K \int_{a_i}^{a_{i+1}} \ddot{F}_\prb(t) \ \diff t
		\\
		                                      & = \sum_{i=0}^K (\dot{F}_\prb^-(a_{i+1}) - \dot{F}_\prb^+(a_{i}))                                                         \\
		                                      & = \dot{F}_\prb^-(a_{K+1}) - \dot{F}_\prb^+(a_{0}) + \sum_{i=0}^{K-1} (\dot{F}_\prb^-(a_{i+1}) - \dot{F}^+_\prb(a_{i+1})) \\
		                                      & \leq \dot{F}_\prb^-(a) - \dot{F}_\prb^+(0).
	\end{align*}
	The last inequality holds since left
	derivatives must be less or equal to right derivatives due to the convexity of $F_\prb$.
	Furthermore, we know due to the convexity of $F_\prbalt \circ \gamma$ that $0 \geq 	\dot{F}_\prbalt^-(a)$ since $\mean$ is the Fr\'echet mean of $\prb$ and $\gamma$ goes from $x$ to $\mu$.

	Since, by definition, $\dot{F}_\prb^+(0) = \nabla_\sigma F_\prb(x)$, Lemma \ref{lemma:conemean} yields the assertion
	$$d(x, \mean) = a \leq - \nabla_\sigma F_\prb (x) \leq - \inf_{\tau \in \Sigma_\mean} \nabla_\tau F_\prb(x)= \widetilde{d}(\mathcal{O},\mathfrak{b}(\widetilde{\prb})).$$

	Write $\mean = \mathfrak{b}(\prb)$. If $\mean = x$, it is trivially true. Thus, assume that $\mean \neq x$. Let $\gamma_{x}^{\mean}: [0, a] \to M$ denote the geodesic from $x$ to $\mean$ (in particular, $a = d(\mean, x)$)
	and let $\sigma = \dir_x (\mean)$. Let $0 < a_1 \leq a_2 \leq \ldots \leq a_K < a$ be the points where $F_\prb$ is not differentiable. We set $a_0 = 0$ and $a_{K+1} = a$

	Write $\ddot{F}_\prbalt = \frac{\diff^2}{\diff t^2} F_\prbalt \circ
		\gamma$ for the second derivatives, wherever they are defined. In
	addition, let $\dot{F}_\prbalt^- =
		\partial_t^- {F}_\prbalt \circ \gamma$ and $\dot{F}_\prbalt^+ = \partial_t^+ {F}_\prbalt \circ \gamma$, where $\partial_t^-$ and $\partial_t^+$ are the left and right derivative of a function, respectively.
	For any $z \in M$ and $t \in (a_i, a_{i+1})$, we have by the
	NPC-inequality (\ref{ineq:NPC}) that the second derivative (computed as
	the symmetric version) of $d^2(z, \gamma(\cdot))$ at $t$ is bounded from
	below as follows:
	\begin{align*}
		\frac{\diff^2}{\diff t^2} d^2(z,\gamma(t)) & = \lim_{h \to 0} \frac{d^2(z, \gamma(t+h)) + d^2(z,\gamma(t-h)) - 2 d^2(z,\gamma(t)) }{h^2} \\
		                                           & \geq \lim_{h \to 0}  \frac{\frac{1}{2}d^2(\gamma(t+h), \gamma(t-h))}{h^2}                   \\
		                                           & \geq \lim_{h \to 0}  \frac{\frac{4}{2} h^2}{h^2}
		= 2,
	\end{align*}
	and therefore $\ddot{F}_\prb\geq 1$.
	We get
	\begin{align*}
		a = \int_0^a 1 \ \diff t & \leq \sum_{i=0}^K \int_{a_i}^{a_{i+1}} \ddot{F}_\prb(t) \ \diff t                                                   \\
		                         & = \sum_{i=0}^K (\dot{F}_\prb^-(a_{i+1}) - \dot{F}_\prb^+(a_{i})                                                     \\
		                         & = \dot{F}_\prb^-(a_{i+1}) - \dot{F}_\prb^+(a_{0}) + \sum_{i=1}^K (\dot{F}_\prb^-(a_{i+1}) - \dot{F}^+_\prb(a_{i+1}) \\
		                         & \leq \dot{F}_\prb^-(a_{i+1}) - \dot{F}_\prb^+(a_{0}).
	\end{align*}
	The last inequality holds since left
	derivatives must be less or equal to right derivatives due to the convexity of $F_\prb$.
	Furthermore, we know due to the convexity of $F_\prbalt \circ \gamma$ that $0 \geq
		\dot{F}_\prbalt^-(a)$ since $\mean$ is the Fr\'echet mean of $\prb$.
	By definition, $\dot{F}_\prb^+(a_{0}) = \nabla_\sigma F_\prb(x)$.
	If follows by Lemma \ref{lemma:conemean} that
	$$d(x, \mean) = a \leq - \nabla_\sigma F_\prb (x) \leq - \inf_{\tau \in \Sigma_\mean} \nabla_\tau F_\prb(x)= \widetilde{d}(\mathcal{O},\mathfrak{b}(\widetilde{\prb})).$$
\end{proof}

\section{A Central Limit Theorem for the Directions of Stickiness}\label{scn:CLT-dir}

Another rather immediate result under the assumptions of Scenario \ref{set:4}
is a central limit theorem for the directional derivatives of the Fr\'echet
function. It follows from fundamental results of
empirical process theory.

\begin{theorem}[Central limit theorem]
	\label{theorem:gaussian}
	In Scenario \ref{set:4}, let $\prb \in \Qkapn[2]$.
	Then for every sample $X_1, \ldots, X_n \iid \prb$, $n\in \NN$, the process
	$\left\{\sqrt{n}\left( \nabla_\sigma F_{\prb_n}(\mean) - \nabla_\sigma
		F_\prb(\mean) \right)\ : \ \sigma \in \Sigma_\mean\right\}$ converges
	weakly, as $n\to \infty$, against a Gaussian process on $\Sigma_\mean$ with zero mean and covariance function
	\begin{align*}
		\cov(\sigma, \tau) & = \int_M \cos(\angle_x(\sigma, z))
		\cos(\angle_\mean(\tau, z))\cdot d^2(\mean,z) \ \diff \prb(z) \quad 	\mbox{ for }\sigma, \tau \in \Sigma_\mean\,.
	\end{align*}
\end{theorem}

\begin{proof}

	By Corollary  \ref{corollary:contfrech}, we see that the family of derivatives of the squared distances to $\mean$ is Lipschitz continuous with respect to the direction. By \cite[Theorem 2.7.11]{vdVaart} for families that are Lipschitz in their parameter, the assumption of Scenario \ref{set:4}
	\begin{align*}
		N\big (\epsilon, \Sigma_{\mean}, \angle_{\mean}\big )
		\leq \frac{V}{\epsilon^k}
	\end{align*}
	yields the bracketing entropy condition of \cite[Theorem 2.5.6]{vdVaart}. The other conditions of that theorem are met as the map $z \mapsto d(\mean,z)$ is an envelope function of $\Phi_{\Sigma_{\mean}}$ and has a finite second moment by the assumption that $\prb \in \mathcal{P}^2(M)$. Hence, that theorem asserts that the family
	$\Phi_{\Sigma_\mean}$ is $\prb$-Donsker. This, in turn implies the convergence against a Gaussian process, see \cite[pp. 81-82]{vdVaart}.

\end{proof}

\section{Outlook}

A main part of this contribution has dealt with stickiness with respect to a single point on \cat~spaces for $\kappa \in \RR$. In Hadamard spaces (complete \catzero~spaces), every $P\in \wst[1]{M}$ has a unique mean, so our sticky flavors did not need to be restricted to suitable subfamilies of $\wst[q]{M}$. In contrast, for $\kappa > 0$, uniqueness of Fr\'echet means can only be guaranteed for $P$ supported by an open geodesic half ball $\Bkap$ and hence, for $\kappa >0$, sticky flavors had to be restricted largely to corresponding families. We conjecture that our results are also valid for all $P\in \wst[1]{M}$, featuring unique means, the arguments would need to be more involved and technical, however, surpassing the scope of this contribution. A starting point for analysis on symmetric spaces may be a generalization of the uniqueness result be \cite{Le98}.
A particular consequence would be that no sticky flavor can be found on finite dimensional differentiable manifolds with curvature bounds.

Flavors of stickiness to a closed set, not necessarily a point has been the subject of Section \ref{scn:3-flavors}. It is worth noting that Theorem \ref{theorem:mainthm} extends at once also to stickiness to the spine in open books (and similar spaces), due to it being a Euclidean product of two Hadamard spaces. Future work may extend these and analog results to set valued stickiness in a wider class of spaces, going beyond stickiness in open books and BHV spaces \cite{bhvsticky}, e.g.  beginning with Riemannian stratified spaces, cf. \cite[p. 328]{HE_Handbook_2020}.


In the literature, also \emph{partial stickiness} (e.g. \cite{HHL+13}) has been considered. Perturbation stickiness extends at once to that case. For topological partial stickiness one requires that every neighborhood contains an open set of sticky distributions, and for directional partial stickiness one would require that $\nabla_\sigma F_\prb(\mu) > 0$ only for some directions $\sigma \in \Sigma_\mu$, cf. \cite{HMMN15}. The generalization to sample and modulation partial stickiness seems more involved, for instance for the former the rescaled limiting distribution would be a nontrivial projection of a Gaussian as is detailed in \cite{tran2023CLT}. The foundations for this rather involved and general result are provided in \cite{tran2023shadow,tran2023geom, tran2023CLTrandom}, making it clear that pursuing partial sticky flavors is again beyond the scope of this contribution.

In practice for statistical analysis of phylogenetic trees, among others, in addition to the BHV space \cite{BHV02}, the \emph{wald space} has been proposed recently by \cite{wald1}. Numerical experiments show that this space, embedded in the \catzero~space of positive definite matrices features unbounded positive and negative curvatures, cf. \cite{wald2}. While our results on basic sticky flavors in Scenario \ref{set:1} apply to wald spaces, it would be interesting, also with respect to applications in phylogenetics to explore which of our results on sticky flavors carry through.

Speaking of applications, one original motivation of this work was to devise a statistical two-sample test for the case that both samples have the same common mean, with respect to which, both exhibit sample stickiness, in effect featuring no asymptotic rescaled variance, rendering the classical two-sample test inapplicable.This has been experienced in \cite{SBH+14}. In this case, as we have shown, the directions of stickiness still feature a nontrivial central limit theorem. Based on this, as well as on modulation convergence, statistical tests can be devised. This is the subject of ongoing research.

In fact, stickiness is not only a phenomenon for Fr\'echet means, it can affect general geometric descriptors of data, for instance a first principal component, as has been observed by \cite{FCOV_MFO_2014} on BHV space. Indeed, the relationships between analog sticky flavors of \emph{generalized Fr\'echet means} \cite{dataannonstan}, comprising among others \emph{principal geodesics} \cite{fletch6}, \emph{geodesic principal components} \cite{HHM07} and \emph{diffusion means} \cite{Nye2011, eltzner2023diffusion}, deserve to be defined and studied. Notably, stickiness of convex generalized Fr\'echet means in metric trees has recently been investigated in \cite{romon2023}, providing a foundation for further study.


\printbibliography
\end{document}